\documentclass[11pt, reqno]{amsart}
\usepackage{amsmath}
\usepackage{amsthm}
\usepackage{amssymb}
\usepackage{amscd,mathrsfs}
\usepackage{amsbsy}
\usepackage{color}
\usepackage{epsfig}
\usepackage{graphicx}
\usepackage{psfrag}
\usepackage[normalem]{ulem}
\usepackage{anyfontsize}
\usepackage{bbm}
\usepackage{hyperref}
\usepackage{changebar}
\usepackage{comment}
\usepackage{bm}

\usepackage[marginpar=.8 in]{geometry}

\usepackage{tikz}
\usetikzlibrary{arrows}
\usetikzlibrary{decorations.pathreplacing}

\usetikzlibrary{calc}
\theoremstyle{plain}
\newtheorem{theorem}{Theorem}[section]
\newtheorem{remark}[theorem]{Remark}
\newtheorem{proposition}[theorem]{Proposition}
\newtheorem{lemma}[theorem]{Lemma}

\newtheorem{definition}[theorem]{Definition}
\newtheorem*{claim}{Claim}

\newtheorem*{lemma*}{Lemma}

\theoremstyle{definition}

\makeatletter
\@namedef{subjclassname@2020}{\textup{2020} Mathematics Subject Classification}
\makeatother

\makeatletter
\newcommand{\labeltext}[2]{%
  \@bsphack
  \csname phantomsection\endcsname 
  \def\@currentlabel{#1}{\label{#2}}%
  \@esphack
}
\makeatother

\newfont\bbf{msbm10 at 12pt}
\def\eps{\varepsilon}

\def\N{{\mathbb N}}

\def\O{{\mathcal O}}

\def\P{{\mathcal P}}

\def\cI{{\mathcal I}}
\def\hcI{\mathring{\cI}}
\def\cJ{{\mathcal J}}
\def\hcJ{\mathring{\cJ}}

\def\cD{{\mathcal D}}

\def\A{{\mathcal A}}

\def\cR{{\mathcal R}}
\def\es{{\emptyset}}
\def\sm{\setminus}

\def\1{\ensuremath{{\mathbbm{1}}}}

\def\orb{\mbox{\rm orb}}

\def\bd{\partial }
\def\le{\leqslant}
\def\ge{\geqslant}

\newcommand{\hDelta}{\mathring{\Delta}}

\newcommand{\Lp}{\mathcal{L}}

\newcommand{\cS}{\mathcal{S}}

\newcommand{\ve}{\varepsilon}

\newcommand{\hLp}{\mathring{\Lp}}

\newcommand{\hF}{\mathring{F}}
\newcommand{\hf}{\mathring{f}}

\newcommand{\hY}{\mathring{Y}}
\newcommand{\he}{\mathring{e}}
\newcommand{\hg}{\mathring{g}}

\newcommand{\holm}{\mathring{m}}

\newcommand{\beq}{\begin{equation}}
\newcommand{\eeq}{\end{equation}}

\newcommand{\be}{\bm{\varepsilon}}

\newcommand{\real}{\text{Re}}
\newcommand{\imag}{\text{Im}}

\DeclareMathOperator*{\essinf}{ess\,inf}

\newcommand{\vertiii}[1]{{\left\vert\kern-0.25ex\left\vert\kern-0.25ex\left\vert #1 
    \right\vert\kern-0.25ex\right\vert\kern-0.25ex\right\vert}}
\newcommand{\invertiii}[1]{{\vert\kern-0.25ex\vert\kern-0.25ex\vert #1 
    \vert\kern-0.25ex\vert\kern-0.25ex\vert}}

\numberwithin{equation}{section}

\begin{document}

\title[Trichotomy for hitting times and escape rates]{A trichotomy for hitting times and escape rates for a class of unimodal maps}
\date{\today}

\begin{abstract}
We consider local escape rates and hitting time statistics for unimodal interval maps of Misiurewicz-Thurston type.  We prove that for any point $z$ in the interval there is a local escape rate and hitting time statistics which is one of three types.  While it is key that we cover all points $z$, the particular interest here is when $z$ is periodic and in the postcritical orbit which yields the third part of the trichotomy.  We also prove generalised asymptotic escape rates of the form first shown by Bruin, Demers and Todd.
\end{abstract}


\thanks{Part of this work was completed during visits of MT to Fairfield University in 2022 and 2023.  MD was partially supported by NSF grant DMS 2055070.  MT was partially supported by the FCT (Funda\c c\~ao para a Ci\^encia  e a Tecnologia) project 2022.07167.PTDC}

\author[M.F.~Demers]{Mark F. Demers}
\address{Department of Mathematics\\
Fairfield University\\
1073 N. Benson Road\\
Fairfield, CT  06824\\
USA} 
\email{mdemers@fairfield.edu}
\urladdr{http://faculty.fairfield.edu/mdemers }

\author[M.~Todd]{Mike Todd}
\address{Mathematical Institute\\
University of St Andrews\\
North Haugh\\
St Andrews\\
KY16 9SS\\
Scotland} 
\email{m.todd@st-andrews.ac.uk}
\urladdr{https://mtoddm.github.io  }

\maketitle

\section{Introduction and main result}
\label{sec:intro}
 
Much attention has been paid recently to the connection between hitting times on the one hand \cite{FreFreTod15, giulietti, atnip} and
escape rates on the other \cite{KL zero, PolUrb18, DemTod21}.  Given a map $f$ of the interval preserving a probability measure $\mu$, 
the principal connection is that in many situations the extremal index, defined by
\[
-\lim_{\ve \to 0} \log \mu( \{ x : f^j(x) \notin B_\ve(z), j=0, 1, \ldots, \lfloor \mu(B_\ve(z))^{-1} \rfloor \} ) ,
\]
where $B_\ve(z)$ denotes the ball of radius $\ve$ centred at $z$,
yields the same limit \cite{BruDemTod18} as the local (or asymptotic) escape rate given by,
\[
\lim_{\ve \to \infty} \lim_{n \to \infty}  \frac{-1}{\mu(B_\ve(z))} \frac{1}{n} \log \mu( \{ x : f^j(x) \notin B_\ve(z), j = 0, 1, \ldots n \} ) \, .
\]
The limit typically obeys a dichotomy:  if $z$ is not periodic, the limit is 1; if $z$ is periodic with prime period $p$, then
the limit is $1 - \frac{1}{|Df^p(z)|}$ if $\mu$ is an acip.  More generally, when $\mu$ is an equilibrium state for a H\"older continuous potential $\phi$, the limit in the periodic case is $1 - e^{S_p\phi(z)}$, where 
$S_p\phi(z) = \sum_{i=0}^{p-1} \phi(f^i(z))$. 

Such dichotomies for asymptotic escape rates have been proved for all $z$ in some uniformly hyperbolic contexts
 \cite{BuYu11,  KL zero, FerPol12, HayYan20}.  
 In the non-uniformly hyperbolic setting, the asymptotic escape rate has been achieved for large sets of $z$ in e.g. \cite{DemTod17, PolUrb18, BruDemTod18, DemTod21}. 
 
In this paper we expand this dichotomy to a trichotomy by establishing conditions in which a third limit
is possible.  We do this principally by considering unimodal maps whose post-critical orbit is eventually periodic and centring holes at such periodic points $z$.  We show that in this case, the asymptotic escape rate is a function of 
both $|Df^p(z)|$ as well as the order of the spike in the invariant density at $z$.

Before stating our main result, we precisely define the class of maps to which our results will apply.

We consider S-unimodal maps $f:[0, 1]\to [0, 1]$ in the sense of \cite{NowSan98}.  That is
\begin{itemize}
\item $f$ is $C^2$ with one critical point $c$;
\item in a neighbourhood of $c$, $f(x)= f(c)- A(x-c)^\ell$ for $\ell>1$ and $A>0$;
\item $|Df|^{-\frac12}$ is convex on each of $[0, c]$ and $[c, 1]$.  
\end{itemize}
\begin{figure}

\includegraphics[scale=0.4]{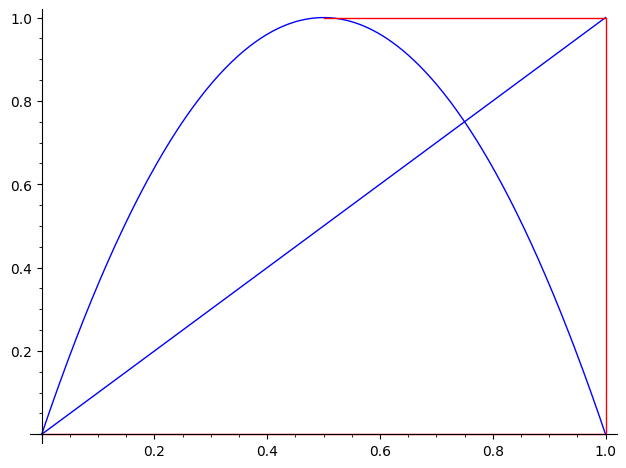}\hspace{2cm}\includegraphics[scale=0.4]{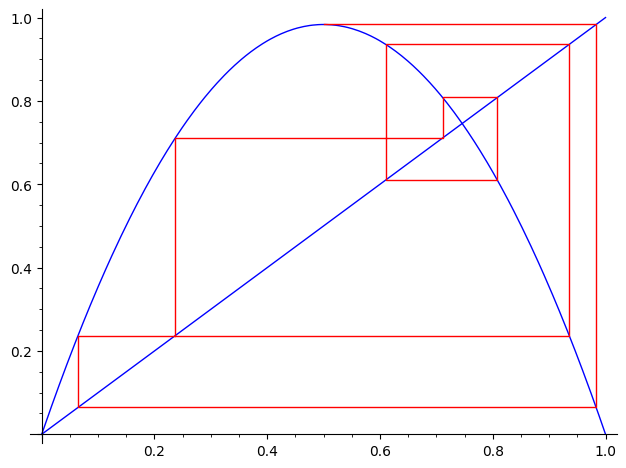}
\caption{Examples of our class of maps in the quadratic family $x\mapsto \frac A4-A(x-1/2)^2$: on the left $A=4$ giving the full quadratic map $x\mapsto 1-4(x-1/2)^2$ where $k_0=2$ and $p=1$; and on the right $A\approx 3.93344...$ where $k_0=3$ and $p=5$.}
\label{fig:egs}
\end{figure}

We assume that the map is Misiurewicz-Thurston: there is a minimal $k_0\ge 2$ such that $f^{k_0}(c)$ is periodic, i.e., there is a minimal $p \ge 1$ such that  $f^p(f^{k_0}(c)) =f^{k_0}(c)$ (see Figure~\ref{fig:egs} for some simple examples).  Moreover this is a repelling periodic point: $|Df^p(f^{k_0}(c))|>1$.  This implies that the postcritical orbit of $c$,  $\orb(f(c))=\{f(c), f^2(c), \ldots, f^{k_0+p-1}(c)\}$, is finite and that there is an absolutely continuous $f$-invariant probability measure (acip) $\mu$.  Then the density of $\mu$ has a spike at each point of $\orb(f(c))$ of type $x^{\frac1\ell -1}$.
 We let $I:=[f^2(c), f(c)]$ denote the \emph{dynamical core}: all points except 0 and 1 eventually map into $I$ and remain there, so $\mu$ is supported in $I$.

Define a hole centred at $z$ by $H_\ve(z) = (z-\ve, z+\ve)$.  The {\em escape rate} of the open system
with hole $H_\ve(z)$ with respect to the measure $\mu$ is defined by
\begin{equation}
\label{eq:escape f}
\mathfrak{e}(H_\ve(z)) = - \lim_{n\to \infty}\frac{1}{n} \log\mu\left(\left\{x\in I: f^j(x)\notin H_\ve(z), \ j=0, \ldots, n-1\right\}\right) ,
\end{equation}
when the limit exists.
Define the \emph{local (asymptotic) escape rate at $z$} by 
\[
\mbox{esc}(z):=\lim_{\eps\to 0} \frac{\mathfrak{e}(H_\ve(z)))}{\mu(H_\eps(z))} \, .
\]
If $z$ is a periodic point with prime period $p$, we write $\lambda_z:= |Df^p(z)|$.

Our main result is the following trichotomy regarding possible values of the local escape rate. 
The only model where such a result has appeared previously is in \cite[Remark 3.11]{DemTod21}, 
which is a very special case (the full quadratic map).

\begin{theorem}
\label{thm:spike}
Let $f$ be as defined above. For any $z\in I$,
\begin{equation*}
{\rm esc}(z) = \begin{cases} 1 & \text{ if $z$ is not periodic},  \\
1-\frac1{\lambda_z} & \text{ if $z$ is periodic and not in orb}(f(c)),\\
1-\frac1{\lambda_z^{1/\ell}} & \text{ if $z$ is periodic and in orb}(f(c)).\\
\end{cases}
\end{equation*}
\end{theorem}

The cases when $z$ is outside the postcritical orbit orb$(f(c))$ follow from \cite[Theorem~3.7]{DemTod21}, so our focus here is on the case when $z\in \mbox{orb}(f(c))$: when 
$z\in \mbox{orb}(f(c))$ is preperiodic (which falls into the first case in the above theorem), and when $z\in \mbox{orb}(f(c))$ is periodic.  The techniques to show the former are essentially a subset of the latter.  
Our proofs exploit the fact that the finite post-critical orbit allows us to define a finite Markov partition
for the map (albeit with unbounded distortion) which we then use to define a first return map to a conveniently
chosen set. 


\subsection{Hitting time statistics}

Setting $r_A(x):=\inf\{k\ge 1:f^k(x)\in A\}$, the \emph{first hitting time to $A$}, the \emph{Hitting Time Statistics (HTS)} at $z$ is given by 
$$T_z(t):=\lim_{\eps\to 0} \mu\left(r_{H_\eps(z)} \ge \frac{t}{\mu(H_\eps(z))}\right)$$
for $t>0$, provided this limit exists.  
A generalisation of this limit is formulated in \cite{BruDemTod18}, where the link between HTS and asymptotic escape rates was explored by scaling the hitting time asymptotic as
$t \mu(H_\ve(z))^{-\alpha}$ for some $\alpha >0$.  
Accordingly, define for $t, \alpha >0$,
\[
L_{\alpha,t}(z) := \lim_{\ve \to 0} \frac{-1}{t \mu(H_\ve(z))^{1-\alpha}} \log \mu \left( r_{H_\ve(z)} \ge \frac{t}{ \mu(H_\ve(z))^{\alpha}} \right) \, .
\]
Note that $T_z(t)$ corresponds to $\alpha =1$.
Using Theorem~\ref{thm:spike} we prove the following in Section~\ref{sec:HTS}.

\begin{theorem}
\label{thm:HTS}
For any $z\in I$ and all $\alpha, t>0$, $L_{\alpha, t}(z)$ exists and equals 
\begin{equation*}
L_{\alpha, t}(z) = \begin{cases} 1 & \text{ if $z$ is not periodic},\\
1-\frac1{\lambda_z} & \text{ if $z$ is periodic and not in }orb(f(c)),\\
1-\frac1{\lambda_z^{\frac1\ell}} & \text{ if $z$ is periodic and in }orb(f(c)).\\
\end{cases}
\end{equation*}
\end{theorem}

As with Theorem~\ref{thm:spike}, the main focus here is on the case when $z \in \mbox{orb}(f(c))$; the other cases follow from adaptations of \cite{BruDemTod18}, as we describe in Section~\ref{ssec:veryend}.

We note that a full description of $T_z(t)$ in uniformly hyperbolic cases can be seen in \cite{FreFreTod12, AytFreVai15}.  In the non-uniformly hyperbolic case, the only full dichotomy (i.e.~the result applies to all $z$ with no
excluded values), and only for $T_z(t)$ rather than $L_{\alpha, t}(z)$ as above,  has been demonstrated for Manneville-Pomeau maps in \cite{FreFreTodVai16}; see also the recent preprint \cite{BanFre23}.  

\begin{remark}
Our techniques actually deal with holes of the form $(z-\eps_L, z+\eps_R)$ where $\frac{\eps_L}{\eps_R}$ is uniformly bounded away from 0 and $\infty$, so we are able to handle non-symmetric holes in both
Theorems~\ref{thm:spike} and \ref{thm:HTS}.  
\end{remark}

\subsection{Structure of the paper}

In Section~\ref{sec:returnstructure} we assume that we are in the $z\in \orb(f(c))$ periodic case and define a domain $Y$ and a first return map $F$ to $Y$ so that the hole at $z$ is outside $Y$:  $(Y, F)$ has exponential return times and a Markov structure.  We outline how domains of $F$ can map into the hole and then modify $F$ to $F_{\be}$ 
by introducing extra cuts at the boundary of the hole. In Section~\ref{sec:holemeas} again we initially assume  $z\in \orb(f(c))$ is periodic and show how the density spike at $z$ affects the scaling of the hole and its preimages.  These scalings lead to the three quantities seen in Theorems~\ref{thm:spike} and \ref{thm:HTS}.   The section ends in Section~\ref{ssec:ratiopreper} where the case $z\in \orb(f(c))$ is preperiodic is dealt with.

Section~\ref{sec:functional} sets up the functional framework for $F_{\be}$ and its punctured counterpart $\hF_{\be}$.  In particular a spectral gap is proved along with the relevant perturbation theory for the transfer operator 
corresponding to $\hF_{\be}$ for $\be$ small and for certain nice `$\beta$-allowable' holes.  The convergence of the relevant spectral properties is proved, which prove a version of Theorem~\ref{thm:spike}, see \eqref{eq:ratio split} and \eqref{eq:limit split} for these specific holes.  Section~\ref{sec:beta approx} then proves Theorem~\ref{thm:spike} for general holes and in all cases of $z\in \orb(f(c))$.  In Section~\ref{sec:HTS} we prove Theorem~\ref{thm:HTS}.  The strategy is to use the induced map to construct a non-Markovian Young Tower and prove a spectral gap for the associated transfer operator in a space of weighted bounded variation.

Notation: we use $g_1(\eps)\gtrsim g_2(\eps)$ to mean $\frac{g_1(\eps)}{g_2(\eps)} \ge h(\eps)$ where $h(\eps)\to 1$ as $\eps\to 0$.  Similarly $g_1(\eps)\lesssim g_2(\eps)$ means $\frac{g_1(\eps)}{g_2(\eps)} \le h(\eps)$ where $h(\eps)\to 1$ as $\eps\to 0$ and $g_1(\eps) \sim g_2(\eps)$ if $\frac{g_1(\eps)}{g_2(\eps)} \to 1$ as $\eps\to 0$.


\section{First return map structure}
\label{sec:returnstructure}

Recall that $k_0 \ge 2$ denotes the minimal integer such that $f^{k_0}(c)$ is periodic with prime period $p$.  
From here until the end of Section~\ref{ssec:ratioper} we will focus on the case that $z\in orb(f(c))$ is periodic; for the preperiodic case we make minor adjustments in Section~\ref{ssec:ratiopreper}.  Suppose that for $k_1\ge k_0$, $f^{k_1}(c)  =z$.  Then necessarily, $f^p(z)=z$ with prime period $p$.

\begin{proposition}
\begin{enumerate}
\item 
 There exists $\lambda_{per}>1$ such that if $x$ is periodic of  period $n$ then $|Df^n(x)| \ge \lambda_{per}^n$.
 \item For each neighbourhood $U$ of $c$ there are $K_U>0$ and $\lambda_U>1$ such that if $J$ is an interval with $U\cap \left(\cup_{i=0}^{n-1} f^i(J)\right) = \es$ then $ \inf_{x \in J} |Df^n(x)|  \ge K_U\lambda_U^n$. 
 \item There is a unique acip $\mu$; this is supported on a finite union of intervals, which contain $c$ in the interior.
 \end{enumerate}
\label{prop:Str}
\end{proposition} 

The first item here is in, for example  \cite[Theorem A]{NowSan98}, while the second is more generally known as Ma\~n\'e's Lemma (eg \cite[Theorem III.5.1]{MelStr93}).  The third item can be found for example as part of \cite[Theorem V.1.3]{MelStr93}.

Recall that we call a hyperbolic periodic point $x$ of (prime) period $n$, \emph{orientation preserving/reversing} if $f^n$ preserves/reverses orientation in small neighbourhoods of $x$ (i.e. $Df^n(x)>0$/$Df^n(x)<0$). 

Denote the orbit of the critical point by $c_i=f^i(c)$.  Then by the definition of $k_0$ and $p$ above, $\text{orb}(c)=\{c_0, \ldots, c_{k_0+p-1}\}$.    Let $\A:= \{A_1, \ldots, A_{M}\}$ for $M=k_0+p-1$ be the set of ordered open intervals in $[f^2(c), f(c)]$ with boundary points from orb$(c)$.  This forms a Markov partition for $f$ on $[f^2(c), f(c)]$.\footnote{Distortion is unbounded.}  The partition induces a dynamical coding on $I':=[f^2(c), f(c)] \sm \bigcup_{n\ge 0}f^{-n}\left(\text{orb}(c) \right)$: to each $x\in I'$ there is a sequence $(x_0, x_1, \ldots )\in \{1, \ldots, M\}^{\N_0}$ with $f^i(x)\in A_{x_i}$.   Let $\Sigma\subset \{1, \ldots, M\}^{\N_0}$ be the corresponding subshift of finite type.
 For each $n\in \N$, we refer to $[x_0, \ldots, x_{n-1}]:= \{(y_0, y_1, \ldots) \in \Sigma: y_0=x_0, \ldots, y_{n-1} = x_{n-1}\}$ as an \emph{$n$-cylinder}, or \emph{cylinder of depth $n$}.  There will be a unique topologically transitive component: we then remove any component which does not intersect this.   The shift is also locally eventually onto and thus topologically mixing.

We will abuse notation and refer to cylinders in $\Sigma$ and the intervals in $I$ they represent  by $w = [x_0, \ldots, x_n]$.  We will write the word 
$w\in  \cup_{n\ge 1} \cup_{w'\in  \{ 1, \ldots, M\}^n }w'$ and the corresponding cylinder $[w]$ as just $w$.
Given cylinders $w_1=[x_0, \ldots, x_n]$ and $w_2=[y_0, \ldots, y_m]$, the concatenation $[w_1w_2]$ is the cylinder 
$[x_0, \ldots, x_n, y_0, \ldots, y_m]$.  When needed, we denote by $\pi$ the projection from $\Sigma$ to $I$.


\subsection{The inducing scheme}
\label{ssec:inducing}
 For $N \in \mathbb{N}$ to be chosen below, we will take a collection of cylinders $w_1, \ldots, w_K\in  \{ 1, \ldots, M \}^N $ containing $z$ and consider the first return map $F=\sigma^\tau$ to $\Sigma\sm  \{w_1, \ldots, w_K\}$ where $\tau$ is the first return time.  We consider the domains of this map to be cylinders $w\in \cup_{n\ge 1} \cup_{w'\in  \{ 1, \ldots, M\}^n }w'$, rather than unions of these.
 Let $\Sigma_{N,K}$ denote the set of one-cylinders for $F$.

\begin{lemma}
Given $K\in \N$, for any $\eta>0$ there exists $N=N(\eta, K) \in \N$ such that for all $n \in \N$, 
$\#\{w \in \Sigma_{N,K} : \tau(w)  =n\}=O(e^{\eta n})$.
\label{lem:count}
\end{lemma}

\begin{proof}
 If $w \in \Sigma_{N,K}$ has $\tau(w) = n$, then $w$ is an $n$-cylinder for $\sigma$ and so must
have the form 
 $[x_0w_{i_1}w_{i_2}\cdots w_{i_j}x_{n-1}]$ where $j\le n/N$.  So we can estimate the number of such cylinders by $K^{\frac nN}$ (note that this estimate comes from considering $(\Sigma, \sigma)$ to be the full shift, which is a substantial over-estimate here).  So the lemma is complete if we choose $N$ large enough that $\frac1N\log K\le \eta$.
\end{proof}

We will use the above partly to ensure our inducing scheme has exponential tails (see Proposition~\ref{prop:scheme}).  We will also choose the scheme to be compatible with the periodic structure at $z$, see Remark~\ref{rmk:spat} below. 

\begin{lemma}
\label{lem:Y}
For any $\ve_0, \eta > 0$, there exist $K, N \in \mathbb{N}$ such that $Y = \pi(\Sigma_{N,K})$ has the following properties:
\begin{enumerate}
  \item[(a)]   $\{ f^i(c) \}_{i=1}^{k_0+p-1} \cap Y = \emptyset$; in particular,  $z \notin Y$;
  \item[(b)] if $x \in I \setminus Y$ and $i \ge 1$ is minimal such that $f^i(x) \in Y$, then $f^i(x) \in (z - \ve_0, z+ \ve_0)$;
  \item[(c)] $f(Y)\supset Y$; 
  \item[(d)] let $F = \sigma^\tau$ denote the first return map to $\Sigma_{N,K}$; then
  $\# \{ w \in \Sigma_{N,K} : \tau(w) = n \} = \mathcal{O}(e^{\eta n})$;
  \item[(e)] $F: \Sigma_{N,K} \circlearrowleft$ is topologically mixing.
\end{enumerate}
\end{lemma}

\begin{proof}
 Fix $\ve_0, \eta >0$. 
For $\tilde N\in \N$, we will choose a pair of $\tilde N$-cylinders $w_L, w_R$ which are adjacent to $z$ (to the left and right of this point respectively) and, observe that for $\tilde N$ large, $f^i(w_L), f^i(w_R)$ are $(\tilde N-i)$-cylinders for $i=1, \ldots, p-1$.  If $i\in \{1, \ldots, p\}$ is minimal with $z=f^{i+k_0}(c)$ then define $w'_L = f^{p-i}(w_L), w_R' = f^{p-i}(w_R)$ and set $w_{j, L}, w_{j, R}$ to be the 
$(\tilde N -p +i +j)$-cylinders along the orbit segment $\{ f(c), \ldots, f^{k_0-1}(c) \}$ 
mapping to $w'_L, w'_R$ by $f^j$ for $j=0, \ldots, k_0-1$.  Then define $Z_{\tilde N}:=\{w_{j, L}, w_{j, R}: j= 0,\ldots, k_0-1\}\cup \{f^i(w_L), f^i(w_R): i=0, \ldots, p-1\}$ be the sets we remove from $\Sigma$.  The deepest cylinders here are either (i) $w_{k_0-1, L}, w_{k_0-1, R}$; or (ii) $w_L, w_R$.  In case (i), these are $N$-cylinders, where $N= \tilde N-p +i +k_0-1$.  Since, for example $w_{k_0-1-i, L}$ consists of at most 
 $M^{i}$ $N$-cylinders (a significant over-estimate), we have removed at most $K=2\frac{M^{k_0+p-1}-1}{M-1}$ $N$-cylinders.  So for $\eta>0$, let  $N_0 =N(\eta, K)\in \N$ be as in Lemma~\ref{lem:count}.  
Set\footnote{Observe that we are removing a neighbourhood of the orbit of the critical \emph{value}: the critical point is not removed.}  $\Sigma'=\Sigma'(N_0):= \Sigma\sm  Z_{\tilde N_0}$  (here $\tilde N_0=N_0+p+i-k_0+1$) and let $F=\sigma^\tau$ be the first return map to $\Sigma'$.   Lemma~\ref{lem:count} implies that the number of $n$-cylinders with first return time $n$ is $O(e^{\eta n})$.  Case (ii) follows similarly.

Define $Y = \pi(\Sigma')$.  Property (a) of the lemma is obvious. The choice of $\tilde N_0$ above guarantees property (d).  The definition
of $Z_{\tilde N_0}$ guarantees that returns to $Y$ must occur in a neighbourhood of $z$ and that $f(Y)\supset Y$, and choosing $\tilde N$ sufficiently large forces this neighbourhood to intersect $(z-\ve_0, z+\ve_0)$, so (b) and (c) hold. 
Finally, (e) follows from the mixing of $\sigma$.
\end{proof}

We will make our final choice of $Y$ once we choose $\ve_0$ before Remark~\ref{rmk:spat}.
We will abuse notation and let $F = f^\tau$ denote the first return map to $Y$ as well.  

Let $\{I_i\}_i$ be the domains (intervals) of monotonicity of the first return map and for brevity, write $\tau_i=\tau|_{I_i}$.  Note that each $I_i$ is associated to a cylinder as in the symbolic model.

\begin{lemma}
There exists $P\in \N_0$ such that $I_i$ contains a periodic point of prime  period $\tau_i+k$ with $k\le P$.
\label{lem:sym_mix}
\end{lemma}

\begin{proof}
Since $I_i$ corresponds to some $n$-cylinder $[x_0w_{i_1}w_{i_2}\cdots w_{i_j}x_{n-1}]$, the lemma is saying that there is a chain of allowable transitions $x_{n-1} \mapsto y_1\mapsto y_{k-1}\mapsto x_0$, for $k\le P$, which follows since the shift is an SFT.
\end{proof}

\begin{proposition}
\label{prop:scheme}
\begin{enumerate}
\item[(a)] $F$ has finitely many images, i.e., $\{F(I_i)\}_i$ is a finite set; 

\item[(b)] (bounded distortion) there exists $C_d>0$ such that if $x,y$ belong to the same $n$-cylinder for $F$, then
\[
\left| \frac{DF^n(x)}{DF^n(y)} - 1 \right| \le C_d |F^n(x) - F^n(y)| \, ;
\]

\item[(c)]  $|\{\tau= n\}| = O\left(e^{-n(\log\lambda_{per}-\eta)}\right)$,  where $| \cdot |$ denotes the
Lebesgue measure of the set. 
\end{enumerate}
\end{proposition}

\begin{proof}
Property (a) follows from the SFT structure and Property (b) follows from the Koebe Lemma added to the fact that we removed a neighbourhood of the orbit of the critical value to create $Y$.

To prove (c), by Lemma~\ref{lem:sym_mix}, given $I_i$, there is a periodic point $x\in I_i$ of period $\tau_i+k$ for $0\le  k \le P$.   Thus, by Proposition~\ref{prop:Str}(1),
\begin{equation}
\label{eq:tilde K}
|Df^{\tau_i}(x)|\ge \lambda_{per}^{\tau_i+k}|Df|_\infty^{-k}\ge\left(\frac{\lambda_{per}}{|Df|_\infty}\right)^{P} \lambda_{per}^{\tau_i}= \tilde K\lambda_{per}^{\tau_i}.
\end{equation}
Since we also have bounded distortion, we see that $|I_i|=O( \lambda_{per}^{-\tau_i})$ and hence applying Lemma~\ref{lem:count}, $|\{\tau= n\}| = O\left(e^{-n(\log\lambda_{per}-\eta)}\right)$, as required.
\end{proof}

In the light of this result, we will assume $\eta\in \left(0,\frac12\log\lambda_{per}\right)$ from here on (we make a final choice for $\eta$ before Remark~\ref{rmk:spat}.

\begin{remark}
The cylinder structure here means that $\left(\cup_{k\ge 1}f^k(\bd Y)\right)\cap \left(Int(Y^c) \sm \text{orb}(f(c))\right)=\es$.  In particular supposing that  $I_i$ has $f^k(I_i) = (x, y)$ in a neighbourhood of $z$ and $y<z$, then there will be a set of domains adjacent to $I_i$ such that the closure of the union of their $f^k$ iterates covers $(x, z)$.  (Similarly if $f^k(I_i)$ lies to the right of $z$.)
\label{rmk:cover}
\end{remark}

 Recall that $\mu$ is the unique acip for $f$ according to Proposition~\ref{prop:Str}. 
Define $\mu_Y:=\frac{\mu|_Y}{\mu(Y)}$, which by Kac's Lemma is an $F$-invariant probability measure.  
We can recover $\mu$ from $\mu_Y$ by the well-known formula, 
\begin{equation}
\label{eq:meas_proj}
\mu(A) = \mu(Y) \sum_i \sum_{j=0}^{\tau_i-1}\mu_Y(I_i \cap f^{-j}(A)), \; \mbox{ for } A \subset I.
\end{equation}

 Note that $f^{k_1}$ in a neighbourhood of $c$ is a 2-to-1 map composed with a diffeomorphism. 
The Hartman-Grobman Theorem implies that $f^p$ restricted to any small enough neighbourhood of $z$, the linearisation domain, is conjugated to the transformation $x\mapsto \lambda_z x$ and is indeed asymptotic to this for $x$ close to $z$.  Let $\eps_0>0$ be such that this theorem holds in $(z-\eps_0, z+\eps_0)$. 

For $\ve \le \ve_0$, let $\delta= \delta(\eps)$ be such that $f^{k_1}|_{(c-\delta, c+\delta)}$ is 2-to-1 onto either $(z-\eps, z]$ or $[z, z+\eps)$.  
Set $\lambda_\delta:= \lambda_{(c-\delta(\eps_0), c+\delta(\eps_0))}$ and $K_\delta:= K_{(c-\delta(\eps_0), c+\delta(\eps_0))}$ from Proposition~\ref{prop:Str}.  Then choose $\eta > 0$ such that
$\eta < \min \{ \frac 12 \log \lambda_{per}, \log \lambda_\delta\}$.
Now with $\ve_0$ and $\eta$ fixed,
we choose $ N_1 \ge N(\eta, K)$ so that the cylinders making up $Z_{\tilde N_1}$  from the proof of
Lemma~\ref{lem:Y}  are contained in $\cup_{i=1}^{p}f^i(z-\eps_0, z+\eps_0)$ and adjust our inducing domain to 
$\Sigma'=\Sigma'(N_1)$, thus also adjusting $Y$.  

\begin{definition}
\label{def:Y}
With $\ve_0$, $\eta$ and $N_1$ fixed as above, we make our final choice of $Y$, which will remain fixed from here on. 
\end{definition}

With $Y$ fixed, we choose $\eps_1\in (0, \eps_0]$ so that $Y\cap\left( \cup_{i=1}^{p}f^i(z-\eps_1, z+\eps_1)\right) = \es$.  

\begin{remark}
\begin{enumerate}
\item
Our choice of cylinders around $\orb(z)$ means that any domains of the first return map to $Y$ can enter $(z-\eps_1, z+\eps_1)$ at most once, since the linear structure of the dynamics in this region means that any domain that does so must be `spat out' into $Y$, and thus have made a return, before it can escape the linearisation domain.
\item 
Moreover, due to the periodic way our cylinders $Z_{\tilde N}$ were chosen, if $f^k(I_i)$ for $k<\tau_i$ is in the neighbourhood of $z$, then it cannot return to $Y$ in less than $p$ steps.  Indeed, $\tau_i-k$ is a multiple of $p$.
\end{enumerate}
\label{rmk:spat}
\end{remark}


It will be convenient to treat the left and right neighbourhoods of $z$ separately.  For this purpose, we introduce
the following notation for asymmetric holes.
Given $\eps_L, \eps_R>0$, we write $\bm{\varepsilon}= \{\eps_L, \eps_R\}$ and define the corresponding hole at $z$ by $H_{\be}(z) = (z-\ve_L, z+\ve_R)$.   We will abuse notation in the following ways: we write $\be=0$ to mean $\eps_L=\eps_R=0$, and we write $\be<\eps$ (also $\be \in (0, \eps)$) to mean $\eps_L, \eps_R<\eps$ (and $\eps_L, \eps_R\in (0, \eps)$).  We extend the definition of $\delta(\eps)$ to $\delta(\be)$ in the natural way, noting that for example in the orientation preserving case $\delta(\be)$ will only depend on one of $\eps_L$ and $\eps_R$.

For notation, we will denote by $(z-\eps_L)_i$, $(z+\eps_R)_i$ the local inverse of $f^{ip}$ applied to $z-\eps_L$, $z+\eps_R$. 

For a hole $H_{\be}(z)$, we denote by 
$H_{\be}'$ the set of intervals $J$ in $Y$ such that $f^s(J)\subset H_{\be}$ for $s<\tau|_J$.  
 Thus $H'_{\be}$ represents the `induced hole' for the first return map $F$.
In fact, due to the construction here, each such $J$ will be one of the domains $I_i$.  The main expression we must estimate is  $\frac{\mu(H_{\be}')}{\mu(H_{\be})}$ for appropriately
chosen $\eps_L, \eps_R$: we will show how this allows us to estimate 
$\frac{\mu(H_\eps')}{\mu(H_\eps)}$ for all sufficiently small $\ve>0$ in Section~\ref{sec:beta approx}.


\subsection{Chain structure}

Let the interval of $I\sm Y$ containing $z$ be denoted by $[a, b]$.  Now suppose $I_i$ is a domain of $Y$ with $f^{s}(I_i)\subset [a, b]$ with $1\le s\le\tau_i$.  Since there is an interval $U$ closer to $z$ than $f^{s}(I_i)$ with $f^p(U) = f^{s}(I_i)$, then by the first return structure, there must be some $I_{i'}$ with $f^{s}(I_{i'})=U$.  Iterating this, we call the resulting domains a \emph{chain}.  We consider $I_i$ here to have some \emph{depth} $k$, and the $I_{i'}$ above to have depth $k+1$ and write these as $I_i^k$ and $I_i^{k+1}$ respectively.  For each chain there is an `outermost' interval which we consider to have depth 1: this is the interval $I_{i''} =I_i^1$, a domain in $Y$ where $f^s(I_{i''}) \subset [a, b]$, but $f^{s+p}(I_{i''})$ not contained in $[a, b]$.  Thus chains are denoted $\{I_i^k\}_{k=1}^\infty$ and $f^{s+kp}(I_{i}^k) =  f^s(I^1_i)$.  

Note that the cylinder structure implies $f^p(a), f^p(b)\notin (a, b)$.

\begin{lemma}
\begin{enumerate}
\item If $f^p$ is orientation preserving around $z$ then two elements $I_i^k$ and $I_i^{k'}$ of a chain have depths differing by 1 if and only if they are adjacent to each other.    Here $f^{s+p}(I_i^{k+1}) = f^s(I_i^k)$ for $s$ as above.
\item If $f^p$ is orientation reversing around $z$ then a chain $\{I_i^k\}_k$  has a neighbouring chain $\{I_j^k\}_k$ so that two elements $I_i^k$ and $I_j^{k'}$ have depths differing by 1 if and only if they are adjacent to each other.  Here the $I_i^k, I_j^{k+1}, I_i^{k+2}$ are adjacent intervals and moreover $f^{2p}(I_i^{k+2})= I_i^k$.
\end{enumerate}
\label{lem:simple_chain}
\end{lemma}

See Figure~\ref{fig:chain} for a sketch of case (1) of the lemma.

\begin{figure}[h]
\begin{tikzpicture}[thick, scale=0.5]

\draw[red,, - ] (-0.5,0) -- (1,0);
\draw[red, loosely dashed ] (-2.5,0) -- (-0.5,0);
\draw (-1,0.3) node[above] {{\small $Y$}};

\draw[ - ] (1,0) -- (27.5,0);
\draw[ loosely dashed ] (27.5,0) -- (29.5,0);
\draw[very thick, - ] (1,-0.4) -- (1,0.4);
\draw (1,0.3) node[above] {{\small $a$}};

\draw[very thick, - ] (21.7,-0.4) -- (21.7,0.4);
\draw (21.5, 0.3) node[above] {{\scriptsize $(z-\eps_L)_1$}};

\draw[very thick, - ] (18.3,-0.4) -- (18.3,0.4);
\draw (18.3,0.3) node[above] {{\scriptsize $z-\eps_L$}};

\draw[very thick, - ] (24.1,-0.4) -- (24.1,0.4);
\draw (24.3,0.3) node[above] {{\scriptsize $(z-\eps_L)_2$}};

\draw[very thick, - ] (27,-0.4) -- (27,0.4);
\draw (27,0.3) node[above] {$z$};

\draw[blue, very thick, (-)]  (22.8,0) -- (24.7,0);
\draw (23.8,-0.2) node[below] {{\scriptsize $f^{s}(I_{i}^k)$}};

\draw[blue, very thick, (-)]  (19.4,0) -- (22.8,0);
\draw (21.1,-0.2) node[below] {{\scriptsize $f^{s}(I_{i}^{k-1})$}};

\draw[blue, very thick, (-)]  (15.4,0) -- (19.4,0);
\draw (17.5,-0.2) node[below] {{\scriptsize $f^{s}(I_{i}^{k-2})$}};

\draw[blue, very thick, (-)]  (1,0) -- (5.6,0);
\draw (3,-0.2) node[below] {{\scriptsize $f^{s}(I_{i}^{1})$}};

\draw[green, ->] (24.0,1.5) parabola[bend pos=0.5] bend +(0,2) (21.8,1.5);
\draw (22.8,3.4) node[above] {$f^{p}$};

\draw[green, ->] (21.5,1.5) parabola[bend pos=0.5] bend +(0,2) (18.5,1.5);
\draw (20.3,3.4) node[above] {$f^{p}$};

\draw[green, ->] (18.3,1.5) parabola[bend pos=0.5] bend +(0,2) (14.2,1.5);
\draw (16.3,3.4) node[above] {$f^{p}$};

\draw[green, ->] (8.5,1.5) parabola[bend pos=0.5] bend +(0,2) (4.5,1.5);
\draw (6.8,3.4) node[above] {$f^{p}$};

\draw[green, thick, loosely dashed](9.5, 2.5)--(13.5, 2.5);

\draw[green, ->] (4.3,1.5) parabola[bend pos=0.5] bend +(0,2) (-0.5,1.5);
\draw (2.8,3.4) node[above] {$f^{p}$};

\end{tikzpicture}
\caption{The chain structure mapped near $z$.  We are assuming that $f^p$ is locally orientation preserving and focussing attention on the left-hand side of $z$.  We sketch some intervals of the images (by $f^s$) of the chain $\{I_{i}^k\}_{k\ge 1}$.  Note that here $\tau|_{I_i^k} = s+k$.}
\label{fig:chain}
\end{figure}

\begin{proof}
We first assume that $\{I_i^k\}_k$ is a chain and that $f^p$ is locally orientation preserving around $z$.  Suppose that $I_j^k$ is adjacent to $I_i^1$.  
Suppose  
for some $t\ge1$ we have  $f^{t-1}(I_i^1)\cap Y =\es$, but $f^{t}(I_i^1)\subset Y$.  As in Remark~\ref{rmk:spat}, according to our choice of $Z_{\tilde N_0}$ in the proof of Lemma~\ref{lem:Y},
$f^t(I^1_i)$ can only reenter $Y$ in an $\ve_0$-neighbourhood of $z$.  Note that $f^{t}(I_i^1)$ must contain either $a$ or $b$  in its boundary: we assume this is $a$.  We observe that $a$ is also a boundary point of $f^{t}(I_j^k)$ and $f^{t}(I_j^k)\cap Y =\es$.  But then we must have $f^{t+1}(I_j^k)\subset Y$, and moreover this domain must contain $a$ in its boundary.  Indeed, by the first return structure it must coincide with $f^{t}(I_i^1)$.  But this implies that actually $I_j^k=I_i^2$.  This means that the left boundary point of $f^{t-1}(I_i^1)$ must map by $f^p$ to the other boundary point $a$.  In the same way we see that all the domains $\{I_i^k\}_k$ are adjacent to each other and have $f^{p}$ mapping adjacent $f^s$-images of the chain to each other.

In the orientation reversing case, suppose that $I_i^1$ is adjacent to $I_j^k$.  Then $f^{s+p}(I_i^1) \subset Y$.  The Markov structure implies that $f^{s+p}(I_j^k) \subset [a, b]$ and that these intervals both contain either $a$ or $b$ in their boundary, which means that $f^{s+2p}(I_j^k) \subset Y$ and so $k=2$.  By the linear structure of $f^{2p}$ in this region, if $I_\ell^k$ is adjacent to $I_j^2$, then $f^{s+2p}(I_\ell^k) = f^s(I_i^1)$, so $I_\ell^k= I_i^3$.  To prove this starting with $I_i^k$ with $k\ge 1$, we iterate the argument by $f^{(k-1)p}$.  The mapping of the $f^s$-images to each other by $f^{2p}$ follows as in the orientation preserving case.
\end{proof}

In case (2) above we call $\{I_i^k\}_k$ and $\{I_j^k\}_k$ \emph{alternating chains}.  We also call any $\{I_i^k\}_{k\ge j}$, for some $j\ge 1$, a \emph{subchain}.

Adding the result of this lemma to the first return structure we see that the images in $[a, b]$ of any chain must be of the same fixed form, i.e., if $\{I_i^k\}_k$, $\{I_j^k\}_k$ are two chains which map by $f^s$ and $f^t$ respectively into $[a, b]$, then assuming $f^s(I_i^k)$ and $f^t(I_j^k)$, for some depth $k\ge 1$, are on the same side of $z$ then these images coincide.  Moreover, $\overline{\cup_k f^s(I_i^k)}$ will be either $[a, z]$ or $[z, b]$ if $f^p$ is orientation preserving around $z$, and $\overline{\cup_k f^s(I_i^k\cup I_j^k)} = [a, b]$ if $f^p$ is orientation reversing around $z$, and the alternating chains are as in the statement of the lemma.


\subsection{Extra cuts for a `Markovian' property}
\label{ssec:Markov} 

Suppose that $f^p$ is orientation preserving around $z$.  Let $a_0= a$ and then inductively let $f^p(a_{i+1})=a_i$, so that $a_i\to z$ as $i\to \infty$.  We similarly define $(b_i)_i$ to the right of $z$, accumulating on $z$.  Observe that domains $I_i^k$ of $F$ which map into $[a, b]$ by some $f^s$ with $1\le s<\tau_i$ must map onto some $(a_{k-1}, a_k)$ or $(b_k, b_{k-1})$.

In the orientation reversing case we define $a_0=a$, and $a_1$ such that $f^p(b)=a_1$ and inductively define $a_{i}$ by $f^{2p}(a_{i+2})= a_{i}$.  Similarly we define $(b_i)_i$.  Here if $f^s(I_i^k)\subset [a, b]$ for $1\le s<\tau_i$ then $f^s(I_i^k) \in \{(a_{k-1}, a_k), (b_k, b_{k-1})\}$.

Note that by bounded distortion, the ratios $\frac{z-a_i}{b_i-z}$  will be comparable to $\frac {z-a}{b-z}$.

The above implies that if we were to take our holes around $z$ as, say $(a_i, b_j)$, these would be Markov holes in the sense that either an interval of the inducing scheme enters the hole before returning to $Y$, or it never intersects it.  Note that by topological transitivity the only way $z$ cannot be accumulated by domains on both the left and the right as above is if $z$ is a boundary point of the dynamical core $[f^2(c), f(c)]$, a simpler case which needs only obvious modifications.

For generic $\eps_L, \eps_R\in (0, \eps_1)$ the corresponding hole will not have this Markov property, so we make an extra cut here, namely if $z-\eps_L\in f^s(I_i)$  (in the sketch in Figure~\ref{fig:chain}, $I_i=I_i^{k-2}$) then we split $I_i$ into two so that one of the resulting intervals maps by $f^s$ into the hole and one maps outside, and similarly for $z+\eps_R$, thus defining $F_{\be}$ from $F=F_0$. 
 For this map, in the generic case, the chains no longer have common boundary points, but now, in the orientation preserving case, they come in pairs $\{I_i^k\}_k, \{\hat I_i^k\}$, where $I_i^{k+1}\bar{\cup} \hat I_i^k$ is a domain of $F$, and $\bar{\cup}$ denotes the union of the open intervals along with their common boundary point.

In the case $f^p$ is orientation reversing around $z$, we may need to cut a domain of $F$ twice to produce $F_{\be}$.  For simplicity, we will not address this case here, but all the ideas below go through similarly.  From now on, unless indicated otherwise, when we discuss our induced map we mean $F_{\be}$.

Remark that due to the extra cuts, although the domains of $F_{\be}$ have the property that each one either maps into $H_{\be}(z)$ or never intersects it, for generic holes this does not define a countable Markov partition for $F_{\be}$.
This is because $F_{\be}(I_i)$ may not be a union of 1-cylinders if a boundary point of $I_i$ is created by one of the
extra cuts. 

We next observe that in the orientation preserving case there is a unique pair of subchains we denote $\{I_{L}^k\}_{k\ge k_{L, \be}}, \{\hat I_{L}^k\}_{k\ge k_{L, \be}} \subset  (c-\delta(\be),  c)$ and a unique pair of subchains $\{I_{R}^k\}_{k\ge k_{R, \be}}, \{\hat I_{R}^k\}_{k\ge k_{R, \be}}\subset  ( c, c+\delta(\be))$.  We call these the \emph{principal subchains}.    Note $\overline{\cup_{k\ge k_{L, \be}}\left(I_{L}^k\cup \hat I_{ I_L}^k\right) \cup \left(\cup_{k\ge k_{R, \be}}I_{R}^k\cup \hat I_{I_R}^k\right)}= [c-\delta(\be), c+\delta(\be)]$.  We can naturally extend this idea to the orientation reversing case. 
In either case, $\delta(\be)$ is determined by only one of $\ve_L$ or $\ve_R$, depending on whether
$f^{k_1}(c - \delta(\be), c+\delta(\be)) = (z - \ve_L, z]$ or $f^{k_1}(c - \delta(\be), c+\delta(\be)) = [z, z + \ve_R)$.
Let us call $\ve_*$ this value in $\{ \ve_L, \ve_R \}$.

 
\section{Measuring the holes}
\label{sec:holemeas}

The aim of this section is to estimate $\mu(H_{\be})$, $\mu(H_{\be}')$ and thus $\frac{\mu(H_{\be}')}{\mu(H_{\be})}$.

\subsection{The contribution from principal subchains}
\label{ssec:principal}

In order to estimate  $\mu(H_{\be}') $ and $\mu(H_{\be})$, we first estimate the contribution from the principal subchains.

\begin{lemma}  For $x$ close to $c$,
$$|x-c|\sim \left(\frac{|f^{k_1}(x)-z|}{A |Df^{k_1-1}(f(c))|}\right)^\frac1\ell.$$
\label{lem:rescale}
\end{lemma}

\begin{proof}
 By the Mean Value Theorem there exists $\theta \in (f(x), f(c))$ such that $|f^{k_1}(x) - f^{k_1}(c)| = Df^{k_1-1}(\theta) |f(x) - f(c)|$.
Then,
\[
\frac{|f^{k_1}(x) - f^{k_1}(c)|}{|x-c|} = \frac{|f(x) - f(c)|}{|x-c|} \frac{|f^{k_1}(x) - f^{k_1}(c)|}{|f(x) - f(c)|}
= \frac{A|x-c|^\ell}{|x-c|} |Df^{k_1-1}(\theta)| \, ,
\]
where $A$ is from Section~\ref{sec:intro}.  The result follows using $|Df^{k_1-1}(\theta)| \sim |Df^{k_1-1}(f(c))|$ by bounded distortion.
\end{proof}

We can see from Lemma~\ref{lem:rescale} that $\left|\overline{\cup_{k\ge k_{L, \be}}\left(I_{L}^k\cup \hat I_{L}^k\right)}\right| \sim C_p (\eps_*)^{\frac1\ell}$  where  $C_p:=\frac{1}{(A|Df^{k_1-1}(f(c))|)^{\frac1\ell}}$,
i.e. $\delta(\be) \sim C_p (\eps_*)^{1/\ell}$. 
Since $c$ is bounded away from the critical orbit, as in \cite[Theorem 6.1]{Mis81} (see also \cite[Theorem A]{Now93}) the invariant density is continuous at $c$ and defining $\rho:=\frac{d\mu}{dm}(c)$, 
the $\mu_Y$-measure of the set can be estimated by, 
\begin{align}\sum_{k\ge k_{L, \be}} \mu_Y(I_{L}^k\cup\hat I_L^k) & \sim \frac{C_p \rho}{\mu(Y)} (\eps_*)^{\frac1\ell}.\label{eq:princH'}
\end{align}

Similarly for $\sum_{k\ge k_{L, \be}} \mu_{Y}(I_{R}^k\cup\hat I_R^k)$.  As in \eqref{eq:meas_proj}, the sum of these two quantities scaled by $\mu(Y)$ is the relevant contribution to $\mu(H_{\be}')$.

Now for the contribution to $\mu(H_{\be})$,
 let us suppose that $f^p$ is locally orientation preserving at $z$ and $f^{k_1-1}$ is also orientation preserving at $f(c)$ (which implies that $\eps_*=\eps_L$), the other cases follow similarly.      Denote  $J^k = \cup_{k' \ge k}  I_L^{k'}\bar\cup \hat I_L^{k'}$.  Then
$f^{k_1}(J^{k_{L, \be}+k}) = ( (z-\ve_L)_{k-1}, z)$ so using Lemma~\ref{lem:rescale} again, we have
\[
\mu_Y(J^{k_{L, \be} + k}) \sim \tfrac{\rho C_p}{\mu(Y)} |(z-\ve_L)_{k} - z|^{\frac{1}{\ell}} \sim \tfrac{\rho C_p}{\mu(Y)} \lambda_z^{-\frac{k}{\ell}} \ve_L^{\frac 1\ell}=  \tfrac{\rho C_p}{\mu(Y)} \lambda_z^{-\frac{k}{\ell}} \ve_*^{\frac 1\ell} \, .
\]
Then using \eqref{eq:meas_proj}, we estimate the (scaled) contribution to $\mu(H_{\be})$ from the left of $c$ by, 
\begin{equation}
\begin{split}
\sum_{k \ge k_{L, \be}} k \mu_Y(I_L^k\cup \hat I_L^k) 
& = \sum_{k \ge k_{L, \be}} \sum_{k' \ge k} \mu_Y(I_L^{k'}\cup\hat I_L^{k'}) = \sum_{k \ge k_{L, \be}} \mu_Y(J^k) \\
& \sim \tfrac{\rho C_p}{\mu(Y)} \ve_*^{\frac 1\ell} \sum_{k \ge 0} \lambda_z^{-\frac{k}{\ell}} 
= \frac{\rho C_\rho \ve_*^{\frac 1\ell} }{\mu(Y)(1 - \lambda_z^{-\frac 1\ell})} \, .
\label{eq:princH}
\end{split}
\end{equation}
An identical estimate follows for the domains $\{I_R^{k'}, \hat I_R^{k'}\}_{k'\ge k}$ from the right of $c$.


\subsection{The contribution from non-principal subchains}

Now for the subchains that are not in $(c-\delta(\be), c+\delta(\be))$, but which map into $H_{\be}$ before returning to $Y$, we will use a different, rougher, type of estimate.
 First notice that such subchains cannot be contained in $(c-\delta(\eps_0), c+\delta(\eps_0))$ since if $I_i\subset (c-\delta(\eps_0), c+\delta(\eps_0))\sm (c-\delta(\be), c+\delta(\be))$, then the repelling structure of our map around $z$ means that $I_i$ will return to $Y$ before mapping into $H_{\be}$.  With this in mind, 
 recall $\lambda_\delta:= \lambda_{(c-\delta(\eps_0), c+\delta(\eps_0))}$ and $K_\delta:= K_{(c-\delta(\eps_0), c+\delta(\eps_0))}$ defined after Remark~\ref{rmk:cover} when we fixed the definition of $Y$.

We will deal with the orientation preserving case here, the orientation reversing case is similar.
Suppose that $\{I_i^k\}_{k\ge k_i}$, $\{\hat I_i^k\}_{k\ge k_i}$ is a pair of non-principal subchains such that $f^s(I_i^k\cup \hat I_i^k)\subset H_{\be}$ for any $k\ge k_i$ (so this is the first time that any element of the subchain enters $H_{\be}$).  For $x\in I_i^k$, by the Mean Value Theorem, there is $x\in I_i^k$ such that
$$|I_i|= |Df^s(x)|^{-1}|f^s(I_i^k)| \le K_\delta^{-1}\lambda_\delta^{-s}|f^s(I_i^k)|$$
by  Proposition~\ref{prop:Str}.  So since $f^s\left(\cup_{k\ge k_i} (I_i^k\cup \hat I_i^k)\right)$ covers either $(z-\eps_L, z)$ or $(z, z+\eps_R)$, we thus obtain an analogue of \eqref{eq:princH'}: for $\eps'= \max\{\eps_L, \eps_R\}$, 
$$\mu(Y) \sum_{k\ge k_i}\mu_Y(I_i^k\cup \hat I_i^k) \le K_\delta^{-1}\lambda_\delta^{-s}\eps'.$$
We know that there are $O(e^{s\eta})$ domains $I_j$ which have $f^s(I_j)=f^j(I_i)$, so, also accounting for the intervals which map to the other side of $z$, the total measure of non-principal chains contributing to $\mu(H_{\be}')$ can be estimated by 
\begin{equation}
\label{eq:nonprincip}
2\sum_{s\ge 1} K_\delta^{-1}e^{s(\eta- \log\lambda_\delta)}\eps' = O(1)\eps',
\end{equation}
and recall we have chosen $\eta\in (0,  \log\lambda_\delta)$:\label{p:eta} in our definition of $Y$.
Similarly to \eqref{eq:princH} we can obtain an analogous estimate, also $O(1)\eps'$, for the contribution to  
$\mu(H_{\be})$.


\subsection{The limiting ratio in the periodic postcritical case}
\label{ssec:ratioper} 

We will assume that $\frac{\eps_L}{\eps_R}$ is uniformly bounded away from 0 and infinity (so in particular $\eps_L, \eps_R = O(\eps_*)$), so that the density spike dominates the measure of our holes.  
Now recalling \eqref{eq:meas_proj}, \eqref{eq:princH'} together with \eqref{eq:nonprincip} imply 
$$\mu(H_{\be}') \sim 2C_p \rho(\eps_*)^{\frac1\ell} + O(1)\eps_* $$
while \eqref{eq:princH} and the analogue of \eqref{eq:nonprincip} imply
$$\mu(H_{\be}) \sim 2C_p \frac{\rho(\eps_*)^{\frac1\ell}}{1-\lambda_z^{-\frac1\ell} }  + O(1)\eps_* ,$$
so the two principal chains dominate this estimate and the limiting ratio is
\begin{equation}
\label{eq:limit ratio}
\lim_{\ve \to 0} \frac{\mu(H_{\be}')}{\mu(H_{\be})} = 1-\lambda_z^{-\frac1\ell} \, .
\end{equation}


\subsection{The preperiodic postcritical case}
\label{ssec:ratiopreper}

For the case that $z=f^{k_1}(c)$ for $1\le k_1< k_0$, we choose the inducing scheme $Y$ via a minor adaptation of the method described in the proof of Lemma~\ref{lem:Y}: properties (a) and (d) of that lemma will follow here along with the condition (b'): If $x\in I\sm Y$ and $i$ is minimal such that $f^i(x)\in Y$ then $f^i(x) \in (f^{k_0+p-1}(c)-\eps_0, f^{k_0+p-1}(c)+\eps_0)$ for some small $\eps_0>0$.  In the notation of the proof of that lemma we choose $w_L$ and $w_R$ adjacent to $z$ (or just one of these if $z\in \{f(c), f^2(c)\}$ is one of the endpoints of $I$) and then choose the rest of the cylinders around orb$(f(c))$, $Z_{\tilde N}:=\{w_{j, L}, w_{j, R}: j= 0,\ldots, k_1-1\}\cup \{f^i(w_L), f^i(w_R): i=0, \ldots, k_0+p-k_1-1\}$ be the sets we remove from $\Sigma$, where $f^{k_1-j}(w_{j, L}) =w_L$ and  $f^{k_1-j}(w_{j, R}) =w_R$.  If $w_L$ and $w_R$ are chosen small enough then properties (a), (c) and (d) of Lemma~\ref{lem:Y} hold, as does Proposition~\ref{prop:scheme}.

This construction also guarantees for small enough $\be>0$ each $I_i$ passes at most once through $(z-\eps_L, z+\eps_R)$ before returning to $Y$, ensuring that ${\mu(H_{\be}')}={\mu(H_{\be})}$.  
Moreover, note that an $I_i$ which does pass through $(z-\eps_L, z+\eps_R)$, must also pass through a neighbourhood of $\{f^{k_1+1}(c), \ldots, f^{ k_0+p-1}(c)\}$ before returning to $Y$.


\section{Functional framework for $\beta$-allowable holes}
\label{sec:functional}

Our goal in this section is to formalise a functional framework for the transfer operator
corresponding to the induced map $F_{\be}$ and its punctured counterpart $\hF_{\be}$.
In order to do this, we will work with a fixed higher iterate $n_0$ of the induced map and
formulate a classification of holes depending on the minimal length of images of $n_0$-cylinders 
under $\hF_{\be}^{n_0}$ (see the definition of $\beta$-allowable in Definition~\ref{def:beta}).
Using this control, we prove the punctured transfer operator enjoys a spectral gap in a space
of functions of bounded variation (Theorem~\ref{thm:gap}) and use it to prove a local escape rate for $\hF_{\be}$
(Lemma~\ref{lem:induced limit}).
We then use this to prove the local escape rate for $f$ needed for Theorem~\ref{thm:spike} via \eqref{eq:ratio split}
by computing the limits in \eqref{eq:limit split}.  
 
We begin by formally defining the induced open system.

Recall that given a hole $H_{\be}(z) = (z-\ve_L, z+\ve_R)$, the induced hole $H_{\be}'$ for $\hF_{\be}$ is 
the collection of intervals that enter $H_{\be}(z)$ before returning to $Y$.  
Define the open system for $n \ge 1$ by
\begin{equation}
\label{eq:define open}
 \hF_{\be}^n = F_{\be}^n|_{\hY_{\be}^n}, \mbox{ where } \hY_{\be}^n := \cap_{i=0}^{n-1} F_{\be}^{-i} (Y \setminus H_{\be}') \, ,
\end{equation}
i.e. the open system at time $n$ is the induced map restricted to the set of points that have not entered $H_{\be}'$ before time $n$.  Note that $\hY^0_{\be} = Y$ and $\hY^1_{\be} = Y \setminus H'_{\be}$.

We first prove the following fact about the map $F_{\be}$.

\begin{lemma}
\label{lem:setup}
For all $n \ge 0$, $|DF_{\be}^n(x)| \ge \tilde K \lambda_{per}^{\tau^n(x)}\ge \tilde K \lambda_{per}^n$,
where $\tilde K >0$ is from \eqref{eq:tilde K}.
\end{lemma}

\begin{proof}
This follows as in the proof of Proposition~\ref{prop:scheme}, using Lemma~\ref{lem:sym_mix} to find a point of the right period.
\end{proof}

Using Lemma~\ref{lem:setup}, we choose $n_0 \ge 1$ so that
\begin{align}
&\inf_x |DF_{\be}^{n_0}(x)| > 3. \label{eq:n0}
\end{align}

Next we define a parameter to keep track of the minimum size of images of $n_0$-cylinders
under the punctured map $\hF_{\be}^{n_0}$. 

\begin{definition}
\label{def:beta}
Let $\beta\in (0,1/2)$ and let $\{ J_i \}_i$ denote the set of one cylinders for $F_{\be}^{n_0}$ ($n_0$-cylinders for $F_{\be}$).

\begin{itemize}
\item For an interval $J= (x, y)$, we say that $t \in J$  is \emph{$\beta$-deep} in $J$ if 
$t \in \left[x+\beta|J|, y-\beta|J|\right]$. 

\item For our holes, we say that $H_{\be}(z)$ is \emph{$\beta$-left-allowable} if there is a domain $J_i$ of $F_{\be}^{n_0}$ with $f^s(J_i) \subset (z-\eps_1,  z) $ with $1\le s<\tau_i$ and $z-\eps_L$ $\beta$-deep in $f^s(J_i)$.  We similarly define \emph{$\beta$-right-allowable} with respect to $(z, z+\eps_1)$. In case $H_{\be}(z)$ satisfies both of these conditions we call it \emph{$\beta$-allowable}. 
\end{itemize}
\end{definition}

The property of $\beta$-allowable is important for the following reason.

\begin{lemma}
\label{lem:beta}
{(Large images depending on $\beta$).}
Let $n_0$ be chosen as in \eqref{eq:n0}.  For each $\beta \in (0,1/2)$, there exists $C_\beta>0$
such that if $H_{\be}$ is $\beta$-allowable, then
  $|F_{\be}^{n_0}(J_i)|, |\hF_{\be}^{n_0}(J_i)| \ge C_\beta$ for all $n_0$-cylinders $J_i$.
\end{lemma}

\begin{proof}
The property follows immediately from the Markov structure of $F^{n_0}_{\be}$ together with the assumption of $\beta$-allowable.  
Note that without extra cuts due to the boundary of the hole, the minimum length of the image of any one-cylinder for
$F^{n_0}_{0}$ is bounded below away from 0 by a number depending only on $\Sigma'(N_0)$.  
Next, by definition of $F^{n_0}_{\be}$, the intervals that are cut by the boundary of the hole
are such that $z - \eps_L$ and  $z+\eps_R$ are $\beta$-deep, by assumption.  Thus the length of the image
under $F_{\be}^{n_0}$ is determined by the parameter $\beta$, together with the distortion constant for $F_{\be}$.
Given our choice of cuts depending on the boundary of the hole, the set of images for $\hF_{\be}^{n_0}$ is
simply a subset of the set of images for $F_{\be}^{n_0}$, so the property holds equally for $\hF_{\be}^{n_0}$. 
\end{proof}


\subsection{A uniform spectral gap for $\beta$-allowable holes}
\label{sec:uniform gap}

In this section we show that for each fixed $\beta>0$ and in any set of $\ve_L, \eps_R>0$ such that $H_{\be}$ 
are $\beta$-allowable, the
transfer operators associated to 
$\hF_{\be}$ and its punctured counterpart have a spectral gap that is uniform in $\beta$ when acting on functions of bounded variation in $Y$.

Given a measurable function $\psi : Y \to \mathbb{R}$, define the \emph{variation} of $\psi$ on an interval $J \subset Y$ (or a finite
collection of intervals $J \subset Y$) by 
\begin{equation}
\label{eq:var def}
\bigvee_J \psi = \sup_{x_0 < x_1 < \cdots < x_N} \sum_{k = 1}^N |\psi(x_k) - \psi(x_{k-1})| \, ,
\end{equation}
where $\{ x_k \}_{k=0}^N$ is the set of endpoints of a partition of $J$ into $N$ intervals, and the supremum ranges over all finite partitions of $J$.
Define the BV norm of $\psi$ by,
\[
\| \psi \|_{BV} = \bigvee_Y \psi + |\psi|_{L^1(m)} \, ,
\]
where $m$ denotes Lebesgue measure on $Y$.
Let\footnote{By the variation of $\psi \in L^1(m)$, we mean the essential variation, i.e.
$\bigvee_Y \psi = \inf_g \bigvee_Y g$, where the infimum ranges over all functions $g$  in the equivalence class
of $\psi$.} $BV(Y)$ denote the set of functions $\{ \psi \in L^1(m) : \| \psi \|_{BV} < \infty\}$.

We shall study the action of the transfer operators associated with $F_{\be}$ and $\hF_{\be}$ acting on $BV(Y)$.
For $\psi \in L^1(m)$ define
\[
\Lp_{\be} \psi(x) = \sum_{y \in F_{\be}^{-1}x} \frac{\psi(y)}{|DF_{\be}(y)|}, \mbox{ and for each $n \ge 0$, } \;
\hLp_{\be}^n \psi = \Lp_{\be}^n( 1_{\hY^n_{\be}} \psi) \, .
\]
We do not claim, or need, that $1_{\hY^n_{\be}} \in BV(Y)$.  

\begin{remark}
\label{rem:same}
Note that for each $x \in Y$, $F_{\be}(x) = F_0(x)$.  This is easy to see since $F_{\be}$ simply introduces extra cuts at the
boundary of $H_{\be}(z)$, but does not change the orbit of $x$, while $F_0$ introduces no extra cuts apart
from those introduced in the original definition of $Y$.  Thus the 1-cylinders for $F_{\be}$ and 
$F_0$ differ
slightly (those for $F_{\be}$ can only be smaller), but pointwise the definition of the maps is the same.
\end{remark}

Our first result proves a uniform set of Lasota-Yorke inequalities for $\hLp_{\be}^{n_0}$, depending only on $\beta$.

Let $\hcI_{\be}$ denote the (countable) collection of one-cylinders for $\hF_{\be}^{n_0}$, and let $\hcJ_{\be}$ denote the finite set of images
of elements of $\hcI_{\be}$.

\begin{proposition}
\label{prop:LY}
For any $\beta>0$ and any $\beta$-allowable hole $H_{\be}(z)$,  for all $\psi \in BV(Y)$ and all
$k \ge 0$,
\begin{eqnarray}
\bigvee \hLp_{\be}^{kn_0} \psi & \le & (\tfrac{2}{3})^k \bigvee_Y \psi +  (1+C_d) \big(C_d + 2 C_\beta^{-1} \big) \sum_{j=0}^{k-1} (\tfrac 23 )^j \int_{\hY^{n_0(k-j)}_{\be}} |\psi| \, dm \, ,
\label{eq:V} \\
\int_Y |\hLp_{\be}^k \psi | \, dm & \le & \int_{\hY^k_{\be}} |\psi| \, dm \, .
\label{eq:L1}
\end{eqnarray}
\end{proposition}

\begin{proof}
The second inequality follows immediately from the definition of $\hLp_{\be}$, so we will prove the first.
In fact, we will prove the inequality for $k=1$, which can then be iterated trivially to produce \eqref{eq:V}.

For $\psi \in BV(Y)$, letting $\{ \bar u_j, \bar v_j \}_j$ denote the endpoints of elements of $\hcJ_{\be}$ and
$\{ u_i, v_i \}_i$ denote the endpoints of elements of $\hcI_{\be}$, we estimate,
\begin{equation}
\label{eq:var split}
\begin{split}
\bigvee_Y \hLp_{\be}^{n_0} \psi & \le \sum_{J_j \in \hcJ_{\be}} \bigvee_{J_j} \hLp_{\be}^{n_0} \psi + \hLp_{\be}^{n_0} \psi(\bar u_j) + \hLp_{\be}^{n_0} \psi(\bar v_j) \\
& \le \sum_{I_i \in \hcI_{\be}} \bigvee_{I_i} \frac{\psi}{|DF^{n_0}_{\be}|} + \sum_{I_i \in \hcI_{\be}} \frac{|\psi(u_i)|}{|DF_{\be}^{n_0}(u_i)|}
+ \frac{|\psi(v_i)|}{|DF_{\be}^{n_0}(v_i)|} \, .
\end{split}
\end{equation}
For the first term above, given a finite partition $\{ x_k \}_{k=0}^N$ of $I_i$, we split the relevant expression into two terms.\,
\[
\sum_k \left| \frac{\psi(x_k)}{|DF_{\be}^{n_0}(x_k)|} - \frac{\psi(x_{k-1})}{|DF_{\be}^{n_0}(x_{k-1})|} \right| 
 \le \frac{1}{3} \bigvee_{I_i} \psi + \sum_k |\psi(x_k)| \left| \frac{1}{|DF_{\be}^{n_0}(x_k)|} - \frac{1}{|DF_{\be}^{n_0}(x_{k-1})|} \right| \, ,
\]
where we have used \eqref{eq:n0} in the first term.  For the second term, we use bounded distortion,
Proposition~\ref{prop:scheme}(b),
to estimate,
\[
\left| \frac{1}{|DF_{\be}^{n_0}(x_k)|} - \frac{1}{|DF_{\be}^{n_0}(x_{k-1})|} \right| 
\le C_d \frac{|F_{\be}^{n_0}(x_k) - F_{\be}^{n_0}(x_{k-1})|}{|DF_{\be}^{n_0}(x_k)|}
\le (1+C_d)C_d |x_k - x_{k-1}| \, ,
\]
where we have applied the mean value theorem to $F_{\be}^{n_0}$ on $[x_{k-1}, x_k]$. 
Putting these estimates together yields,
\[
\begin{split}
\sum_k \left| \frac{\psi(x_k)}{|DF_{\be}^{n_0}(x_k)|} - \frac{\psi(x_{k-1})}{|DF_{\be}^{n_0}(x_{k-1})|} \right| 
& \le \frac 13 \bigvee_{I_i} \psi + C_d(1+C_d) \sum_{k=1}^N |\psi(x_k)| (x_k - x_{k-1}) \\
& \le \frac 13 \bigvee_{I_i} \psi + C_d(1+C_d) \int_{I_i} |\psi| + \kappa_N(\psi) \, ,
\end{split}
\]
where we have recognised the second term as a Riemann sum, and the error term $\kappa_N(\psi) \to 0$ as $N \to \infty$.  Since
the variation is attained in the limit of partitions as $N \to \infty$, we have the following bound on the first term from
\eqref{eq:var split},
\begin{equation}
\label{eq:first var term}
\bigvee_{I_i} \frac{\psi}{|DF_{\be}^{n_0}|} \le \frac 13 \bigvee_{I_i} \psi + C_d(1+C_d) \int_{I_i} |\psi| \, dm \, .
\end{equation}

Next, for the second term in \eqref{eq:var split}, we use the bound,
\begin{equation}
\label{eq:sum}
|\psi(u_i)| + |\psi(v_i)| \le 2 \inf_{I_i} |\psi| + \bigvee_{I_i} \psi \le \frac{2}{|I_i|} \int_{I_i} |\psi| + \bigvee_{I_i} \psi \, .
\end{equation}
Then using again bounded distortion together with Lemma~\ref{lem:beta}, we have 
\begin{equation}
\label{eq:large image}
|I_i| \inf_{I_i} |DF_{\be}^{n_0}| \ge (1+C_d)^{-1} |F_{\be}^{n_0}(I_i)| \ge (1+C_d)^{-1} C_\beta \, .
\end{equation}
Putting these estimates together with \eqref{eq:first var term} into \eqref{eq:var split}, and using again
\eqref{eq:n0}, we conclude,
\[
\bigvee_Y \hLp_{\be}^{n_0} \psi \le \frac 23 \sum_i \bigvee_{I_i} \psi  + \big(C_2(1+C_d) + 2(1+C_d)C_\beta^{-1}\big) \int_{I_i} |\psi|
\le \frac 23 \bigvee_Y \psi + (1+C_d)(C_d + 2C_\beta^{-1}) \int_{\hY^{n_0}_{\be}} |\psi| ,
\]
which is the required inequality for $k=1$.
\end{proof}

Next, in order to show that $\hLp_{\be}$ has a uniform spectral gap (depending on $\beta$, and for $\be$ sufficiently small),
we will apply the perturbative framework of Keller and Liverani \cite{KL pert}.  To this end, define the
norm of an operator $\P : BV(Y) \to L^1(m)$ by,
\begin{equation}
\label{eq:tri}
||| \P ||| = \sup \{ |\P \psi |_{L^1(m)} : \| \psi \|_{BV} \le 1 \} \, .
\end{equation}
Our next lemma is standard:  $||| \Lp_0 - \hLp_{\be} |||$ is small as a function of $m(H_{\be}')$.

\begin{lemma}
\label{lem:small pert}
For any $\be>0$ and $\be' \in [0, \be)$, $||| \hLp_{{\be '}} - \hLp_{\be} ||| \le m(H_{\be}' \setminus H'_{{\be '}}) $.
This holds in particular for ${\be '}=0$, in which case $\hLp_0 = \Lp_0$ is the unpunctured operator.
\end{lemma}

\begin{proof}
Let $\psi \in BV(Y)$.
Then,
\[
\int_Y |(\hLp_{{\be '}} - \hLp_{\be}) \psi | \, dm \le \int_Y |\psi 1_{H_{\be}' \setminus H'_{{\be '}}} | \, dm \le \| \psi \|_{BV} m(H_{\be}' \setminus H'_{{\be '}}) \, .
\]
\end{proof}

\begin{theorem}
\label{thm:gap}
For any $\beta>0$, there exists $\ve_\beta>0$ such that for any $\beta$-allowable hole $H_{\be}(z)$ with 
$\be < \ve_\beta$, $\hLp_{\be}$ is a continuous perturbation of $\Lp_0$.  Indeed, it is H\"older
continuous in $m(H_{\be}')$.

As a consequence, $\hLp_{\be}$ has a spectral gap on $BV(Y)$.  In particular,  there exist $\eta_\beta, B_\beta >0$,
such that for all $\be<\ve_\beta$, the spectral radius of $\hLp_{\be}$ is $\Lambda_{\be} <1$ and there exist
operators $\Pi_{\be}, \cR_{\be}: BV(Y) \circlearrowleft$ satisfying $\Pi_{\be}^2 = \Pi_{\be}$, $\Pi_{\be} \cR_{\be} = \cR_{\be} \Pi_{\be} = 0$,
and  $\| \cR_{\be}^n \|_{BV} \le B_\beta \Lambda_{\be}^n e^{-\eta_\beta n}$  such that
\begin{equation}
\label{eq:decomp}
\hLp_{\be} \psi  = \Lambda_{\be} \Pi_{\be} \psi + \cR_{\be} \psi \, .
\end{equation}
Moreover, $\Pi_{\be}  = \he_{\be} \otimes \hg_{\be}$ for some $\he_{\be} \in BV(Y)^*$ and $\hg_{\be} \in BV(Y)$ satisfying
$\hLp_{\be} \hg_{\be} = \Lambda_{\be} \hg_{\be}$ with $\int_Y \hg_{\be} \, dm = 1$.

Lastly, the spectra and spectral projectors vary (H\"older) continuously as functions of $\be$ in the
$||| \cdot |||$-norm, i.e. as operators from $BV(Y)$ to $L^1(m)$.
\end{theorem}

\begin{proof}
The Lasota-Yorke inequalities of Proposition~\ref{prop:LY} apply also to the unpunctured operator
$\Lp_{\be} = \Lp_0$ with $\hY^n_{\be}$ replaced by $Y$.  Thus $\Lp_0^{n_0}$ is quasi-compact on $BV(Y)$, and since $\Lp_0$ is also bounded as an operator on $BV(Y)$ (although we do not obtain the same contraction for one iterate 
of $\Lp_0$, the norm estimate as in the proof of Proposition~\ref{prop:LY} is finite), 
then $\Lp_0$ is also quasi-compact on $BV(Y)$.
Since $F_0$ is mixing by Lemma~\ref{lem:Y}(e) and has finite images by Proposition~\ref{prop:scheme}(a), then $F_0$ is covering
in the sense of \cite{LSV}.  It follows that $\Lp_0$ has a spectral
gap.  Then so does $\Lp_0^{n_0}$.  Moreover, if $g_0$ is the unique element of $BV(Y)$
such that $\Lp_0 g_0 = g_0$ and $\int_Y g_0 \, dm = 1$, then \cite[Theorem~3.1]{LSV} implies
\begin{equation}
\label{eq:lower g}
c_g := \inf_Y g_0 > 0 \, .
\end{equation}

Next, due to the uniform Lasota-Yorke inequalities (for fixed $\beta>0$) of Proposition~\ref{prop:LY} together with
Lemma~\ref{lem:small pert} for ${\be '}=0$, \cite[Corollary~1]{KL pert} implies that the spectra and spectral projectors of $\hLp_{\be}^{n_0}$ outside the disk of radius $2/3$
vary continuously in $\be$ for $\be$ sufficiently small (depending on $\beta$). 
The spectral gap and the rest of the spectral decomposition for $\hLp_{\be}$ then follows from the analogous decomposition for $\Lp_0$.  

Lastly, fixing $\be < \ve_\beta$ and using Lemma~\ref{lem:small pert}, we apply again \cite[Corollary~1]{KL pert} to 
$\hLp_{\be}$ to conclude that the spectra and spectral projectors of 
$\hLp_{{\be '}}$ vary H\"older continuously as functions of $|{\be '} - \be|$ whenever $\be' < \ve_\beta$.
\end{proof}

The above theorem implies in particular that the size of the spectral gap is at least $\eta_\beta$ for all $\beta$-allowable holes $H_{\be}(z)$ with $\be < \ve_\beta$.


\subsection{Local escape rate}

In this section, we will set up the estimates necessary to prove Theorem~\ref{thm:spike} via the induced
map $F_{\be}$.  The strategy is essentially the same as that carried out in \cite[Section~7]{DemTod21}, but of course
now we are interested in the case in which $z$ lies in the critical orbit, which was not allowed for 
geometric potentials in \cite{DemTod21}.  

We fix $\beta >0$ and consider the zero-hole limit as $\ve\to 0$ for $\beta$-allowable holes only.
As a first step, we use the spectral gap for $\hLp_{\be}$ given by Theorem~\ref{thm:gap}
to construct an invariant measure $\nu_{\be}$ for
the induced open map $\hF_{\be}$ supported on the survivor set $\hY^\infty_{\be} := \cap_{n=1}^\infty \hY^n_{\be} = \cap_{n=0}^\infty \hF_{\be}^{-n}(Y)$.

Define for $\psi \in BV(Y)$,
\begin{equation}
\label{eq:nu def}
\nu_{\be}(\psi) := \lim_{n\to\infty} \Lambda_{\be}^{-n} \int_{\hY^n_{\be}} \psi \, \hg_{\be} \, dm \, .
\end{equation}

\begin{lemma}
\label{lem:nu}
Fix $\beta>0$ and let $\be < \ve_\beta$ be such that $H_{\be}$ is $\beta$-allowable.
The limit in \eqref{eq:nu def} exists and defines a Borel probability measure, supported on $\hY^\infty_{\be}$, and
invariant for $\hF_{\be}$.  Moreover, $\nu_{\be}$ varies continuously as a function of $\be$ (for fixed $\beta$)
and \footnote{Indeed, we show that $\nu_{\be}(\tau)$ is continuous in $\be$ although $\tau \notin BV(Y)$.}
\begin{equation}
\label{eq:escape relation}
-\log \Lambda_{\be} = \left( \int \tau \, d\nu_{\be} \right) \mathfrak{e}(H_{\be}(z)) \, ,
\end{equation}
where $\tau$ is the inducing time for $F_{\be}$ and $\mathfrak{e}(H_{\be}(z))$ is the escape rate for $f$ from
\eqref{eq:escape f}.
\end{lemma}

\begin{proof}
The limit in \eqref{eq:nu def} exists due to the spectral decomposition from Theorem~\ref{thm:gap} and the
conformality of $m$:
\[
\lim_{n \to \infty} \Lambda_{\be}^{-n} \int_{\hY^n_{\be}} \psi \, \hg_{\be} \, dm = \lim_{n \to \infty} \int_Y \Lambda_{\be}^{-n} \hLp_{\be}^n(\psi \hg_{\be}) \, dm = \he_{\be}(\psi \hg_{\be}) \, ,
\]
for any $\psi \in BV(Y)$ since if $\psi \in BV(Y)$, then also $\psi \hg_{\be} \in BV(Y)$.  
From \eqref{eq:nu def}, 
$| \nu_{\be}(\psi) | \le \nu_{\be}(1) |\psi|_\infty$, so that 
$\nu_{\be}$ extends to a bounded linear functional on $C^0(Y)$, i.e.
$\nu_{\be}$ is a Borel measure, clearly supported on $\hY_{\be}^\infty$.  Since $\nu_{\be}(1)=1$, $\nu_{\be}$ is a probability measure.

Next, we prove that $\nu_{\be}$ is continuous as an element of $BV(Y)^*$.
Remark that by the above calculation,
$\nu_{\be}(\psi) = \he_{\be}(\hg_{\be} \psi)$ for $\psi \in BV(Y)$, and when $\ve=0$,
$\mu_Y(\psi) = e_0(g_0 \psi) = \int g_0 \psi \, dm$, since $m$ is conformal for $\Lp_0$.  
Indeed, $\he_{\be}$ defines a conformal measure $\holm_{\be}$ for $\hLp_{\be}$, so that $\he_{\be}(\psi) = \int_Y \psi \, d\holm_{\be}$ and $d\nu_{\be} = \hg_{\be} d\holm_{\be}$.
Thus, for
$\psi \in BV(Y)$ and $\be, \be' < \ve_\beta$,
\begin{equation}
\label{eq:nu cont}
\begin{split}
| \nu_{\be}(\psi) - \nu_{{\be '}}(\psi) | & \le | \he_{\be}(\hg_{\be} \psi - \hg_{\be'} \psi) | + | \he_{\be}(\hg_{\be'} \psi) - e_{\be'}(\hg_{\be'} \psi)| \\
& \le \left| \int (\hg_{\be} - \hg_{\be'} ) \psi \, d\holm_{\be} \right| + \left| \int \big( \Pi_{\be}(\hg_{\be'} \psi) - \Pi_{\be'}(\hg_{\be'} \psi) \big) \, dm \right| \\
& \le |\psi|_\infty | \hg_{\be} - \hg_{\be'} |_{L^1(\holm_{\be})} + ||| \Pi_{\be} - \Pi_{\be'} ||| \, \| \hg_{\be'} \psi \|_{BV} \, .
\end{split}
\end{equation}
Both differences go to 0 as $\be' \to \be$ by Theorem~\ref{thm:gap}, while $\| \hg_{\be'} \|_{BV}$ is uniformly 
bounded in $\be'$ by Proposition~\ref{prop:LY}.
We conclude that $\nu_{\be}$ is continuous in $\be$ for fixed $\beta$ when acting on $BV$ functions.  

It remains to prove \eqref{eq:escape relation}.
Unfortunately, 
$\tau \notin BV(Y)$  so first we must show that $\nu_{\be}(\tau)$ is well-defined.  
Indeed, it is easy to check that $\hLp_{\be} \tau \in BV$.  This holds since
$\tau$ is constant on each 1-cylinder $Y_{i,\be}$ for $F_{\be}$.  Thus $\hLp_{\be} \tau$ has discontinuities
only at the endpoints of $\{ \hF_{\be}(Y_{i, \be}) \}_i $, which
is a finite collection of intervals.

It follows also that $\hLp_{\be} (\tau \hg_{\be}) \in BV(Y)$.  Thus using \eqref{eq:decomp},
\[
\begin{split}
\lim_{n \to \infty} \Lambda_{\be}^{-n} & \int_{\hY^n_{\be}} \tau \hg_{\be} \, dm
= \lim_{n \to \infty} \Lambda_{\be}^{-n} \int_Y \hLp_{\be}^{n-1}(\hLp_{\be}(\tau \hg_{\be})) \, dm \\
& = \Lambda_{\be}^{-1} \int \Pi_{\be} (\hLp_{\be}(\tau \hg_{\be})) \, dm + 
\lim_{n \to \infty} \int \Lambda_{\be}^{-n} \cR_{\be}^{n-1} ( \hLp_{\be}(\tau \hg_{\be}) ) \, dm \, ,
\end{split}
\]
and the second term converges to 0 by Theorem~\ref{thm:gap}.  Thus the limit defining
$\nu_{\be}(\tau)$ exists and is uniformly bounded in $\be$ for fixed $\beta$.  More than this,
the above calculation can be improved to show that $\tau$ is uniformly (in $\be$) integrable with respect to $\nu_{\be}$,
as follows.
For each $N>0$, we use the above to estimate,
\[
\nu_{\be}(1_{\tau > N} \cdot \tau) = \lim_{n \to \infty} \Lambda_{\be}^{-n} \int \hLp_{\be}^{n-1}(\hLp_{\be}( 1_{\tau>N} \cdot \tau \hg_{\be})) \, dm
\le \Lambda_{\be}^{-1} | \hLp_{\be}( 1_{\tau>N} \cdot \tau \hg_{\be})|_\infty \, .
\]
Then using bounded distortion as in \eqref{eq:DF}, one has for $x \in Y$,
\[
\begin{split}
|  \hLp_{\be}( 1_{\tau >N} \cdot \tau \hg_{\be})(x) | & \le \frac{(1+C_d)|\hg_{\be}|_\infty}{C_\beta} \sum_{\substack{y \in \hF_{\be}^{-1}x \\ \tau(y) > N}} \tau(y) m(Y_{i, \be}(y)) \\
& \le C \sum_{k>N} k \, m(\tau=k) \le C' \sum_{k > N} k e^{-k(\log \lambda_{per} - \eta)} \, ,
\end{split}
\]
where we have used Proposition~\ref{prop:scheme}(c).

It follows that for each $\kappa>0$, there exists $N>0$ such that $\sup_{\be \in [0, \ve_\beta)} \{ \nu_{\be}(1_{\tau > N} \cdot \tau) \} < \kappa$
where the sup is restricted to values of $\be$ such that $H_{\be}$ is $\beta$-allowable. 
Let $\tau^{(N)} = \min \{ \tau, N \}$ and note that $\tau^{(N)} \in BV(Y)$.  
Then taking a limit along $\be'$ corresponding to $\beta$-allowable $H_{\be'}$ yields,
\[
\lim_{\be' \to \be} | \nu_{\be}(\tau) - \nu_{{\be '}}(\tau)| \le \lim_{\be' \to \be} | \nu_{\be}(\tau^{(N)}) - \nu_{{\be '}}(\tau^{(N)})| + |\nu_{\be}(1_{\tau>N} \cdot \tau)|
+ |\nu_{{\be '}}(1_{\tau>N} \cdot \tau)| \le 2\kappa \, ,
\]
since we have shown that $\nu_{\be'} \to \nu_{\be}$ as elements of $BV(Y)^*$.  Since $\kappa>0$ was arbitrary, this
proves that $\nu_{\be}(\tau)$ varies continuously in $\be$.

Recall that $\{ Y_i \}_i$ denotes the set of 1-cylinders for $F_0$ and let $\cJ_0 = \{ F_0(Y_i) \}_i$ denote the finite
set of images.  The covering property of $f$ implies that the set of preimages of endpoints of elements
of $\cJ_0$ is dense in $I$.
Thus there is a dense set of $\be < \ve_\beta$ such that
$F_{\be}$ admits a countable Markov partition with finite images.  For such $\be$,  
\cite[Section~6.4.1]{DemTod17} implies that $\nu_{\be}$ is an equilibrium state for the potential 
$- \Xi_{\be} \cdot \log (DF_{\be}) - \log \Lambda_{\be}$, where
$\Xi_{\be}(x) = 1$ if $x \in Y \setminus H_{\be}'$ and $\Xi_{\be} = \infty$ if $x \in H_{\be}'$. 
Similarly, \cite[Lemma~5.3]{BDM} implies
that $\nu_{\be}$ is a Gibbs measure for the potential 
$- \Xi_{\be} \cdot \log (DF_{\be}) - \tau \mathfrak{e}(H_{\be}(z))$ with pressure equal to 0.
Putting these together yields \eqref{eq:escape relation} for such `Markov holes.'  

Finally, we extend the relation to all $\be$ via the continuity of $\Lambda_{\be}$ and $\nu_{\be}(\tau)$.
Since $\mathfrak{e}(H_{\be}(z))$ is monotonic in $\be$ and equals $-\log \Lambda_{\be}/\nu_{\be}(\tau)$ on a dense set
of $\be$, it must also be continuous in $\be$.  Thus the relation \eqref{eq:escape relation} holds for all $\be < \ve_\beta$
for which $H_{\be}$ is $\beta$-allowable.  
\end{proof}

With Lemma~\ref{lem:nu} in hand, we see that the limit we would like to compute to prove
Theorem~\ref{thm:spike} can be expressed as follows, 
\begin{equation}
\label{eq:ratio split}
\frac{\mathfrak{e}(H_{\be}(z))}{\mu(H_{\be}(z))} = \frac{- \log \Lambda_{\be}}{\mu_Y(H_{\be}')} \cdot \frac{\int \tau \, d\mu_Y}{\int \tau \, d\nu_{\be} } \cdot \frac{\mu(H_{\be}')}{\mu(H_{\be}(z))} \, ,
\end{equation}
where as before $\mu_Y = \frac{\mu|_Y}{\mu(Y)}$, and $1/\mu(Y) = \int_Y \tau \, d\mu_Y$ by Kac's Lemma since $F_\ve$ is a first-return map to $Y$.  Theorem~\ref{thm:spike} will follow once
we show that as $\be \to 0$,
\begin{equation}
\label{eq:limit split}
\frac{- \log \Lambda_{\be}}{\mu_Y(H_{\be}')}  \to 1 \, \quad \frac{\int \tau \, d\mu_Y}{\int \tau \, d\nu_{\be} } \to 1 \, ,
\quad \frac{\mu(H_{\be}')}{\mu(H_{\be}(z))} \to 1 - \lambda_z^{-1/\ell} \, . 
\end{equation}

The third limit above is precisely \eqref{eq:limit ratio} when $z \in \mbox{orb}(f(c))$ is periodic
(and is simply 1 by Section~\ref{ssec:ratiopreper} when $z \in \mbox{orb}(f(c))$ is preperiodic)
, so we proceed to prove the first two.
We first prove these limits for $\beta$-allowable holes with a fixed $\beta>0$ 
in Lemmas~\ref{lem:induced limit} and \ref{lem:return int}.  Then in Section~\ref{sec:beta approx}
we show how
to obtain the general limit as $\be \to 0$.

\begin{lemma}
\label{lem:induced limit}
For fixed $\beta>0$ and any sequence of $\beta$-allowable holes $H_{\be}(z)$,
\[
\lim_{\be \to 0} \frac{- \log \Lambda_{\be}}{\mu_Y(H_{\be}')} = 1 \, .
\]
\end{lemma}

\begin{proof}
The lemma could follow using the results of \cite{KL zero}, yet since $z \notin Y$ in our setting, it is not clear how to 
verify the aperiodicity of $H_{\be}'$ without imposing an additional condition on the reentry of points to $Y$
which have spent some time in a neighbourhood of $z$.  To avoid this, we will argue directly, as in
\cite[Proof of Theorem~7.2]{DemTod21}, yet our argument is simpler since our operators $\hLp_{\be}$ approach a
fixed operator $\Lp_0$ via Lemma~\ref{lem:small pert} in contrast to the situation in \cite{DemTod21}.

We assume that $\be < \ve_\beta$ so that we are in the setting of Theorem~\ref{thm:gap}.  Then
since $g_0 \in BV(Y)$, iterating
\eqref{eq:decomp} yields for any $n \ge 1$,
\[
\hLp_{\be}^n g_0 = \Lambda_{\be}^n \he_{\be}(g_0) \hg_{\be} + \cR_{\be} g_0 \implies
\hg_{\be} = \frac{1}{\he_{\be}(g_0)} ( \Lambda_{\be}^{-n} \hLp_{\be}^n g_0 - \Lambda_{\be}^{-n} \cR_{\be}^n g_0 ) \, .
\]
Using this relation and Lemma~\ref{lem:small pert} yields,
\begin{equation}
\label{eq:reduce}
\begin{split}
1 - \Lambda_{\be} & = \int \hg_{\be} \, dm - \int \hLp_{\be} \hg_{\be} \, dm
= \int (\Lp_0 - \hLp_{\be}) \hg_{\be} \, dm = \int_{H_{\be}'} \hg_{\be} \, dm \\
& = \frac{1}{\he_{\be}(g_0)} \int_{H_{\be}'} ( \Lambda_{\be}^{-n} \hLp_{\be}^n g_0 - \Lambda_{\be}^{-n} \cR_{\be}^n g_0 ) \, dm \\
& = \frac{1}{\he_{\be}(g_0)} \left( \int_{H_{\be}'} g_0 \, dm - \int_{H_{\be}'} (1-\Lambda_{\be}^{-n} \hLp_{\be}) g_0 \, dm
- \int_{H_{\be}'} \Lambda_{\be}^{-n} \cR_{\be}^n g_0 \, dm \right) \, .
\end{split}
\end{equation}
By Theorem~\ref{thm:gap}, the third term on the right side of \eqref{eq:reduce} is bounded by
\[
\int_{H_{\be}'} \Lambda_{\be}^{-n} \cR_{\be}^n g_0 \, dm \le B_\beta e^{- \eta_\beta n} \| g_0 \|_{BV} c_g^{-1} \int_{H_{\be}'} g_0 \, dm \, ,
\]
where $c_g = \inf_Y g_0 > 0$ according to \eqref{eq:lower g}.

Next, the second term on the right side of \eqref{eq:reduce} can be expressed as,
\[
\int_{H_{\be}'} (1-\Lambda_{\be}^{-n} \hLp_{\be}) g_0 \, dm
= (1-\Lambda_{\be}^{-n}) \int_{H_{\be}'} g_0 \, dm + \Lambda_{\be}^{-n} \int_{H_{\be}'} (\Lp_0^n - \hLp_{\be}^n) g_0 \, dm \, ,
\]
using the fact that $\Lp_0 g_0 = g_0$.  We claim that $(\Lp_0^n - \hLp_{\be}^n) g_0$ can be made small in
$L^\infty(Y)$.  To see this, choose $n = kn_0$ and write for $x \in Y$,
\[
(\Lp_0^n - \hLp_{\be}^n) g_0(x) = \sum_{y \in F_0^{-n} x} \frac{g_0(y) 1_{Y \setminus \hY^n_{\be}}(y)}{|DF_0^n(y)|}
= \sum_{j=0}^{k-1} \sum_{y \in F_0^{-n}x} \frac{g_0(y) 1_{\hY_{\be}^{jn_0} \setminus \hY_{\be}^{(j+1)n_0}}(y)}{|D F_0^n(y)|} \, ,
\]
where we have used the fact that when $n = kn_0$, 
$Y \setminus \hY_{\be}^n = \cup_{j=0}^{k-1} (\hY_{\be}^{jn_0} \setminus \hY_{\be}^{(j+1)n_0})$ and
the union is disjoint.
Next, if $Y_{i,j, \ve}(y)$ is the $jn_0$-cylinder for $F_{\be}^{jn_0}$ containing $y$, then by bounded distortion and Lemma~\ref{lem:beta},
\begin{equation}
\label{eq:DF}
|DF^{jn_0}(y)| \ge \frac{m(F^{jn_0}(Y_{i,j, \ve}(y))}{(1+C_d)m(Y_{i,j, \ve}(y))} \ge \frac{C_\beta}{(1+C_d)m(Y_{i,j,\ve}(y))} \, .
\end{equation}
Then, applying also \eqref{eq:n0}, we have
\[
(\Lp_0^n - \hLp_{\be}^n) g_0(x) \le \frac{(1+C_d)|g_0|_\infty}{C_\beta} \sum_{j=0}^{k-1} m(\hY^{jn_0}_{\be} \setminus \hY^{(j+1)n_0}_{\be}) 3^{k-j} =: \rho_n(\be) \, ,
\]
and note that $\rho_n(\be) \to 0$ as $\be \to 0$ for fixed $n = kn_0$.  Thus, the second term on the right side of \eqref{eq:reduce}
is estimated by,
\[
\int_{H_{\be}'} (1-\Lambda_{\be}^{-n} \hLp_{\be}) g_0 \, dm
\le \big(1- \Lambda_{\be}^{-n} + c_g^{-1} \Lambda_{\be}^{-n} \rho_n(\be) \big) \mu_Y(H_{\be}').
\]

Putting these estimates into \eqref{eq:reduce} yields,
\[
\frac{1- \Lambda_{\be}}{\mu_Y(H_{\be}')}
= \frac{1}{\he_{\be}(g_0)} \left( 1 + \mathcal{O}( e^{-\eta_\beta n} + (1-\Lambda_{\be}^{-n}) + \Lambda_{\be}^{-n} \rho_n(\be) \right).
\]
Fixing $\kappa>0$, first choose $n = kn_0$ so that $e^{-\eta_\beta n} < \kappa$.  Next, choose
$\be$ so small that by the continuity of the spectral data from the proof of Theorem~\ref{thm:gap},
$|1-\Lambda_{\be}^{-n}| < \kappa$, $\Lambda_{\be}^{-n} \le 2$, $\rho_n(\be) < \kappa$ and 
$|\he_{\be}(g_0) -1| < \kappa$.  This last bound is possible since $e_0(g_0) = \int g_0 \, dm = 1$.  Thus the
relevant expression is $1 + \mathcal{O}(\kappa)$ for $\be$ sufficiently small and $H_{\be}$ $\beta$-allowable.  Since
$\kappa$ is arbitrary, the lemma is proved.
\end{proof}

Recall that the inducing time $\tau$ for $F_{\be}$ does not depend on $\be$.

\begin{lemma}
\label{lem:return int}
For fixed $\beta>0$ and any sequence of $\beta$-allowable holes $H_{\be}(z)$,
\[
\lim_{\be \to 0} \frac{\int \tau \, d\mu_Y}{\int \tau \, d\nu_{\be} } = 1 \, .
\]
\end{lemma}

\begin{proof}
As above, we assume $\be < \ve_\beta$.  Recall that $\nu_0 = \mu_Y$ and $e_0(\psi) = \int \psi \, dm$ since
$m$ is conformal for $\Lp_0$.  Thus the estimates \eqref{eq:nu cont} and following in the proof
of Lemma~\ref{lem:nu} hold equally well with ${\be '}=0$ throughout.  Thus the continuity
of $\nu_{\be}(\tau)$ extends to $\be =0$.  This, plus the fact that $\nu_{\be}(\tau) \ge 1$ since $\tau \ge 1$ implies the
required limit.
\end{proof}


\section{Completion of the proof of Theorem~~\ref{thm:spike} via approximation}
\label{sec:beta approx}

Putting together Section~\ref{ssec:ratioper} and Lemmas~\ref{lem:induced limit} and \ref{lem:return int}, we have
proved Theorem~\ref{thm:spike} for each $\beta>0$ along sequences $(\be_n)_n$ where each $H_{\be_n}$ is $\beta$-allowable.
It remains to consider the alternative case where we have to approximate 
 a given sequence $H_{\be_n}$  by $\beta$-allowable $H_{\be_n'}$.  We will focus our estimates on approximating on the left, with the right-hand side following similarly.  We also start by assuming $z\in \mbox{orb}(f(c))$ is periodic.

We remark that if $H_{\be}$ is $\beta$-allowable, then it is also $\beta'$-allowable for any $\beta' < \beta$, 
so we will take our approximating sequence with $\beta$ tending to 0.  Without loss of generality and for convenience, we assume
$\beta < (2\lambda_z)^{-1}$.

Recall the notation $(a_i)_i, (b_i)_i$ from Section~\ref{ssec:Markov},  i.e. $(a_i, a_{i+1})$
are the $f^s$-images of subchains of intervals accumulating on $z$ from the left, while $(b_{i+1}, b_i)$ accumulate
on $z$ from the right.  Recall that $\frac{z-a_i}{b_i-z}$ are uniformly bounded away from 0 and $\infty$.  As in Lemma~\ref{lem:simple_chain},
we have $f^p(a_i, a_{i+1}) = (a_{i-1}, a_i)$ if $f^p$ is orientation preserving at $z$, and
$f^{2p}(a_i, a_{i+1}) = (a_{i-2}, a_{i-1})$ if $f^p$ is orientation reversing.  We will assume the 
orientation preserving case in the following, the orientation reversing case being similar.

Each interval $(a_i, a_{i+1})$ is $f^s(I_j)$ for some one-cylinder $I_j$ for $F_{\be}$.  Each
$I_j$ in turn is subdivided into a countable union of $n_0$-cylinders for $F_{\be}$
and so $(a_i, a_{i+1})$ is subdivided into a countable union of $f^s$ images of these $n_0$-cylinders.
Indeed, given the chain structure, the arrangement of $f^s$ images of the $n_0$-cylinders
in $(a_i, a_{i+1})$ maps precisely under $f^p$ onto the $f^s$ images of the $n_0$-cylinders
in $(a_{i-1}, a_i)$.  

We describe this structure as follows:  each $(a_i, a_{i+1})$ is subdivided into finitely many 
intervals (depending on $n_0$), which we index by $j$, $j = 1, \ldots, J_i$.  The end points of these intervals are
preimages of the boundaries of one-cylinders in $Y$, some of which may map onto $c$ before 
$n_0$ iterates of $F_\ve$,
in which case they are in fact accumulation points of preimages of the chains recalled above.
We label these chains of intervals $(c_{i,j,k}, c_{i,j,k+1})$ which accumulate on the $j$th point
in $(a_i, a_{i+1})$ from the left, and similarly $(d_{i,j,k+1}, d_{i,j,k})$ accumulate from the right.

Set $v_{i,j,k} = c_{i,j,k+1} - c_{i,j,k}$.  According to our definition, if $H_\ve$ is $\beta$-left-allowable, 
then 
\[
z - \ve \in A_{L, \beta} := \cup_{i,j,k} [c_{i,j,k}+ \beta v_i, c_{i,j,k+1} - \beta v_i] \, .
\]
 So if $H_\ve$ is not $\beta$-left-allowable, then $z - \ve \in (c_{i,j,k+1} - \beta v_{i,j,k}, c_{i,j,k+1} + \beta v_{i,j,k})$ 
 for some $i,j,k$.  Remark that if there is no accumulation of chains at the $j$th point, then the index $k$ is redundant.
 Also, there is some duplication since $c_{i,0} = a_i$, while $c_{i,J_i} = a_{i+1}$, yet for uniformity of notation we 
 denote each range of non-$\beta$-left-allowable values of $z-\ve$ as above.
 
We approximate $\ve$ from above by
$\ve_o^L := z - (c_{i,j,k+1} - \beta v_{i,j,k})$ and from below by
$\ve_u^L := z - (c_{i,j,k+1} + \beta v_{i,j,k})$.  Both $H_{\ve_o^L}$ and $H_{\ve_u^L}$ are $\beta$-left-allowable
 and $\ve \in (\ve_u^L, \ve_o^L)$. 
Thus,
\[
\frac{\ve_u^L}{\ve_o^L} \le \frac{\ve_u^L}{\ve} \le \frac{\ve_o^L}{\ve} \le \frac{\ve_o^L}{\ve_u^L} \, ,
\] 
so we focus on estimating the outer two quantities.  Notice also that the above implies $\frac{\ve_u^L}{\ve_u^R}, \frac{\ve_o^L}{\ve_o^R}$ are  uniformly bounded away from 0 and $\infty$, due to the uniform bound on $\frac{z-a_i}{b_i-z}$, so we may apply \eqref{eq:limit ratio}.

Now,
\[
\frac{\ve_u^L}{\ve_o^L} = \frac{z - (c_{i,j,k+1} + \beta v_{i,j,k})}{z - (c_{i,j,k+1} - \beta v_{i,j,k})}
= 1 - \frac{2 \beta v_{i,j,k}}{z - (c_{i,j,k} - \beta v_{i,j,k})} \, .
\]
Note that $z - (c_{i,j,k+1} - \beta v_{i,j,k}) \ge z - a_{i+1}$ while $v_{i,j,k} \le e_i := a_{i+1} - a_i$.
But due to the chain structure around $z$, we have $z - a_{i+1} > e_{i+1} \sim \lambda_z^{-1} e_i$, so that
\begin{equation}
\label{eq:ratio bound}
\frac{\ve_u^L}{\ve_o^L} \ge 1 - \frac{2 \beta e_i}{e_{i+1}} \gtrsim 1 - 2 \beta \lambda_z \, .
\end{equation}
In the same way, $\frac{\ve_o^L}{\ve_u^L} \lesssim 1 + 2 \beta \lambda_z$.  The right hand estimates for
the analogous $\ve_u^R$ and $\ve_o^R$ are similar.

 To complete the proof of Theorem~\ref{thm:spike}, we must evaluate the following limit as $\ve \to 0$,
\begin{equation}
\label{eq:final limit}
\frac{1}{\mu(z-\eps, z+\eps)}\frac{-1}n \log\mu\left(\left\{x\in I: f^i(x)\notin (z-\eps, z+\eps), \ i=1, \ldots, n\right\}\right) \, .
\end{equation}
We estimate this from below by
$$\frac{\mu(z-\eps_u^L, z+\eps_u^R)}{\mu(z-\eps, z+\eps)} \frac{1}{\mu(z-\eps_u^L, z+\eps_u^R)}\frac{-1}n \log\mu\left(\left\{x\in I: f^i(x)\notin (z-\eps_u^L, z+\eps_u^L), \ i=1, \ldots, n\right\}\right).$$
By the above estimates, 
$\frac{\mu(z-\eps_u^L, z+\eps_u^R)}{\mu(z-\eps, z+\eps)} \sim (1- 2\beta \lambda_z)^{\frac1\ell}$, 
so taking the limit as $\ve \to 0$ yields an lower bound of
\[
( 1 - 2\beta \lambda_z)^{\frac1\ell} \cdot \left( 1 - \lambda_z^{-1/\ell} \right) \, ,
\] 
where the second factor comes from the application of Theorem~\ref{thm:spike} to $\beta$-allowable holes
via \eqref{eq:limit split}.  Similarly, one obtains an upper bound for \eqref{eq:final limit} of
$( 1 + 2\beta \lambda_z)^{\frac1\ell} \cdot \left( 1 - \lambda_z^{-1/\ell} \right)$.
Since these bounds hold for all sufficiently small $\beta$, we take $\beta \to 0$ to obtain the required
limit for Theorem~\ref{thm:spike} along an arbitrary sequence $( \ve_n )_n$
in the case that $z \in \mbox{orb}(f(c))$ is periodic.

Finally notice that if $z\in \mbox{orb}(f(c))$ is preperiodic then the above calculations all go through similarly: from the construction of $(Y, F)$ in Section~\ref{ssec:ratiopreper}, the periodic structure of the postcritical orbit can be pulled back to $z$ to generate the $(a_i)_i, \ (b_i)_i$ required, but our bounds are of the form $( 1 \pm 2\beta \lambda)$ where $\lambda = |Df^p(f^{k_0}(z))|$.
These tend to 1 as $\beta \to 0$, yielding the required limit in the case that $z$ is preperiodic, but not periodic.


\section{Proof of Theorem~\ref{thm:HTS}}
\label{sec:HTS}

In this section, we prove the cases of Theorem~\ref{thm:HTS} in several steps.
First, we address the case where $z \in \mbox{orb}(f(c))$.
As in the proof
of Theorem~\ref{thm:spike}, we first fix $\beta >0$ and consider sequences of holes
$H_{\be}$ that are $\beta$-allowable. 
Leveraging the existence of the induced map $F_0$ and the local
escape rate proved in Theorem~\ref{thm:spike}, we prove Theorem~\ref{thm:HTS} for sequences
of such holes in Section~\ref{sec:HTS beta}.  Then in Section~\ref{sec:beta approx two},
we will let $\beta \to 0$ to prove the required hitting time statistics for all sequences
$\ve \to 0$.

Finally, in Section~\ref{ssec:veryend} we show how to adapt the results of
\cite{BruDemTod18} to prove the hitting time statistics when $z \notin \mbox{orb}(f(c))$.


\subsection{A non-Markovian tower with a hole}
\label{sec:nonMarkov tower}

Throughout Sections~\ref{sec:nonMarkov tower} -- \ref{sec:beta approx two} we will assume
$z \in \mbox{orb}(f(c))$.

Recall the induced map $F_0^{n_0}$, where $n_0$ is chosen so that
\eqref{eq:n0} holds.  
Indeed, by the proof of Proposition~\ref{prop:scheme} (see also
Lemma~\ref{lem:setup}), there exists $\gamma>0$ such that
\begin{equation}
\label{eq:gamma}
\inf_{x \in Y} |Df^{\tau^{n_0}(x)}(x)| \, e^{-\gamma \tau^{n_0}(x)} > 2.5 \, , 
\end{equation}
where $\tau^{n_0} = \sum_{k = 0}^{n_0-1} \tau \circ F_0^k$.

Using $F_0^{n_0}$, we define an extended system that resembles a Young tower as follows.
Define
\[
\Delta = \{ (y, k) \in Y \times \mathbb{N} : k \le \tau^{n_0} -1 \} .
\]
We refer to the $k$th level of the tower as $\Delta_k = \{ (y, n) \in \Delta : n = k \}$.  Sometimes
we refer to $(x, k) \in \Delta_k$ as $x$ if $k$ is clear from context.

The dynamics on the tower is defined by $f_\Delta(y,k) = (y,k+1)$ if $k < \tau^{n_0} -1$ and
$f_\Delta(y,\tau^{n_0}(y) -1) = (F_0^{n_0}(y), 0)$.
There is a natural projection $\pi_\Delta : \Delta \to I$ which satisfies $f \circ \pi_\Delta = \pi_\Delta \circ f_\Delta$.
Clearly, $\pi_\Delta(\Delta_0) = Y$. 
Remark that $F_0$, $f_\Delta$ and $\Delta$ depend on $z$, but not on $\be$.

It follows from Lemma~\ref{lem:Y}(c) and the fact that $f$ is locally eventually onto that $f_\Delta$
is mixing.  Indeed, $f(Y) \supset Y$ and the chain structure around $z$ implies $f^{k_0+p}(Y) \supset I$.
Then since $f$ is locally eventually onto, for any interval $A \subset Y$, there exists 
$n_A \in \mathbb{N}$ such that $f^{n_A}(A) \supset Y$, and again by Lemma~\ref{lem:Y}(c),
$f^{n_A+k}(A) \supset Y$ for all $k \ge 0$.  Setting $A_0 = (\pi_\Delta|_{\Delta_0})^{-1}(A)$, this implies
$f_\Delta^{n_A+k}(A_0) \supset \cup_{j=0}^k \Delta_j$ for each $k \ge 0$.  Thus $f_\Delta$ is
topological mixing.\footnote{Indeed, the above argument also implies that there exist $x, y \in A_0$ such that
$\tau^{n_0}(x) = n_A$ and $\tau^{n_0}(y) = n_A + 1$ so g.c.d.$(\tau^{n_0})=1$, yet this is not a sufficient
condition for $f_\Delta$ to be mixing since our tower has multiple bases.}

We lift Lebesgue measure to $\Delta$ simply by defining 
$m_\Delta|_{\Delta_0} = m|_Y$ and $m_\Delta|_{\Delta_k} = (f_\Delta^k)_*(m_\Delta|_{\Delta_0})$.
Note that by Proposition~\ref{prop:scheme}, $m_{\Delta}(\tau^{n_0} > n) \le C e^{- \zeta n}$,
where $\zeta = \log \lambda_{\text per} - \eta > 0$ and $\eta$ is chosen before 
Remark~\ref{rmk:cover}.  Thus our tower has exponential tails. 

Now recall $\alpha \in (0, \infty)$ from the statement of Theorem~\ref{thm:HTS}.  This will determine the
scaling at which we consider the hitting time to $H_{\be}(z)$.
We reduce $\gamma >0$ in \eqref{eq:gamma} if necessary so that
\begin{equation}
\label{eq:gamma two}
\gamma < \zeta \; \mbox{ and } \; \gamma/\zeta < \alpha \, .
\end{equation}
Note that the second condition is relevant only if we consider $\alpha < 1$.

Given a hole $H_{\be}(z)$, define $H_{\Delta_k}(\be) = (\pi_\Delta|_{\Delta_k})^{-1}(H_{\be}(z)$ and
$H_\Delta(\be) = \cup_{k \ge 1} H_{\Delta_k}$.  The corresponding punctured tower is defined analogously to
\eqref{eq:define open} for $n \ge 1$,
\[
\hf_{\Delta, \be}^n = f_{\Delta}^n|_{\hDelta_{\be}^n}, \mbox{ where } \hDelta_{\be}^n := \cap_{i=0}^{n-1} f_\Delta^{-i}(\Delta \setminus H_{\Delta}(\be)) \, .
\]

The main difference between $f_\Delta : \Delta \circlearrowleft$ and the usual notion of a Young
tower is that we do not define a Markov structure on $\Delta$.
Yet as demonstrated above, the tower has the following key properties:
\begin{itemize}
  \item Exponential tails: $m_{\Delta}(\tau^{n_0}> n) \le C e^{-\zeta n}$;
  \item $f_\Delta$ is topologically mixing; 
    \item Large images at returns to $\Delta_0$: if $H_{\be}(z)$ is $\beta$-allowable, then $C_\beta>0$ from Lemma~\ref{lem:beta} provides a lower bound on the length of images in returns to $\Delta_0$ under both $f_\Delta$ and $\hf_{\Delta, \be}$. 
\end{itemize}
We will use these properties to prove the existence of a uniform spectral gap for the punctured transfer operators on the 
tower acting on a space 
of functions of weighted bounded variation, as follows.

 The inducing domain $Y$ is comprised of finitely many maximal connected components that
 are intervals in $I$.  Let $\widetilde{\cD}_0$ denote this collection of intervals.  Similarly, for each $k \ge 1$,
 $f^k(Y \cap \{ \tau^{n_0} \ge k+1 \})$ can be decomposed into a set of maximal connected components.
 We further subdivide these intervals at the first time $j \le k$ that they contain 
 a point in orb$(f(c))$, $z - \ve_L$ or
 $z + \ve_R$.  These are the cuts introduced in the definition of $F_\ve$ in 
 Section~\ref{ssec:Markov}.
 Let $\widetilde{\cD}_k$ denote this collection of intervals and define the collection of lifts by,
 \[
 \cD_k = \{ J = (\pi|_{\Delta_k})^{-1}(\tilde J) : \tilde J \in \widetilde{\cD}_k \} \, , \mbox{ for all $k \ge 0$}.
 \]
 
For an interval $J \in \cD_k$ and a measurable function $\psi: \Delta \to \mathbb{R}$, define 
$\bigvee_J \psi$, the variation of $\psi$ on $J$ as in \eqref{eq:var def}.  On each level
$\Delta_k$, $k \ge 1$, define
\[
\| \psi \|_{\Delta_k} = e^{-\gamma k} \sum_{J \in \cD_k} \bigvee_J \psi + e^{-\gamma k} \sup_{\Delta_k} |\psi| \, .
\]
Finally, define the weighted variation norm on $\Delta$ by,
\[
\| \psi \|_{WV} = \sum_{k \ge 0} \| \psi \|_{\Delta_k} \, .
\]
If $\| \psi \|_{WV} < \infty$ then $| \psi |_{L^1(m_\Delta)} < \infty$ since $\gamma < \zeta$ by \eqref{eq:gamma two}.
So we denote by $WV(\Delta)$ the set of functions $\psi \in L^1(m_\Delta)$ such that 
$\| \psi \|_{WV} < \infty$.  Since we consider $WV(\Delta)$ as a subset of $L^1(m_\Delta)$,
we define the variation norm of the equivalence class to be the infimum of variation norms
of functions in the equivalence class. 

We define the transfer operator $\Lp_\Delta$ corresponding to $f_\Delta$ acting on $L^1(m_\Delta)$ in the natural way,
\[
\Lp_\Delta \psi(x) = \sum_{y \in f_\Delta^{-1}(x)} \frac{f(y)}{Jf_\Delta(y)} \, , 
\]
where $Jf_\Delta$ is the Jacobian of $f_\Delta$ with respect to $m_\Delta$.
Then the transfer operator for the punctured tower is defined for each $n \ge 1$ by,
\[
\hLp_{\Delta, \be}^n \psi = \Lp_\Delta^n (1_{\hDelta_{\be}^n} \psi) \, .
\]

Our goal is to prove the following proposition, which is the analogue of Theorem~\ref{thm:gap}, but for the tower
map rather than the induced map.

\begin{proposition}
\label{prop:tower gap}
For any $\beta>0$, there exists $\ve_\beta(\Delta)$ such that for any $\beta$-allowable hole $H_{\be}(z)$ with
$\be < \ve_\beta(\Delta)$, both $\Lp_\Delta$ and  $\hLp_{\Delta, \be}$ have a spectral gap on $WV(\Delta)$.

Specifically, there exist $\sigma_\beta, A_\beta >0$ such that for all $\beta$-allowable holes $H_{\be}$ with 
$\be < \ve_\beta(\Delta)$,
the spectral radius of $\hLp_{\Delta, \be}$ is $\lambda_{\Delta, \be} < 1$ and there exist operators
$\Pi_{\Delta, \be}, \cR_{\Delta, \be} : WV(\Delta) \circlearrowleft$ satisfying
$\Pi_{\Delta, \be}^2 = \Pi_{\Delta, \be}$, $\Pi_{\Delta, \be} \cR_{\Delta, \be} = \cR_{\Delta, \be} \Pi_{\Delta, \be} = 0$,
and $\| \cR_{\Delta, \be}^n \|_{WV} \le A_\beta \lambda_{\Delta, \be}^n e^{-\sigma_\beta n}$ such that
\[
\hLp_{\Delta, \be} \psi = \lambda_{\Delta, \be} \Pi_{\Delta, \ve} \psi + \cR_{\Delta, \be} \psi\, ,
\mbox{ for all $\psi \in WV(\Delta)$.}
\]
Indeed, $\Pi_{\Delta, \be} = \he_{\Delta, \be} \otimes \hg_{\Delta, \be}$ for some $\he_{\Delta, \be} \in WV(\Delta)^*$
and $\hg_{\Delta, \be} \in WV(\Delta)$ 
satisfying $\hLp_{\Delta, \be} \hg_{\Delta, \be} = \lambda_{\Delta, \be} \hg_{\Delta, \be}$ and
$\int_{\Delta} \hg_{\Delta, \be} \, dm_\Delta = 1$.
\end{proposition}

The proof of this proposition is based on the following sequence of lemmas, which prove
the compactness of $WV(\Delta)$ in $L^1(m_\Delta)$, uniform Lasota-Yorke inequalities
for $\hLp_{\Delta, \be}$, and the smallness of the perturbation when viewing
$\Lp_{\Delta} - \hLp_{\Delta, \be}$ as an operator from $WV(\Delta)$ to $L^1(m_\Delta)$.

\begin{lemma}
\label{lem:compact}
The unit ball of $WV(\Delta)$ is compactly embedded in $L^1(m_\Delta)$.
\end{lemma}
\begin{proof}
If $\| \psi \|_{WV} \le 1$, then restricted to each $J \in \cD_k$, 
the variation of $\psi$ is at most $e^{\gamma k}$ and $|\psi|_J|_\infty \le e^{\gamma k}$.  
Thus if $B_1$ is the ball of radius 1
in $WV(\Delta)$, then $B_1|_J$ is compactly embedded in $L^1(m_{\Delta}|_J)$.

Taking a sequence $(\psi_n)_n \subset B_1$, we first enumerate the elements of 
$\cup_{k \ge 0} \cD_k$
and then use compactness on each $J$ and a diagonalization procedure to extract a subsequence 
$(\psi_{n_k})_k$ which converges on every $J \in \cup_{k \ge 0} \cD_k$
to a function $\psi$.  $\psi$ necessarily belongs to $L^1(m_\Delta)$ due to dominated convergence
since $|\psi_n|_{\Delta_k}|_\infty \le e^{\gamma k}$ and the function which is equal to $e^{\gamma k}$ on 
$\Delta_k$ for each $k$ is integrable
since $\gamma < \zeta$.
\end{proof}

\begin{lemma}
\label{lem:tower LY}
Assume $H_{\be}(z)$ is $\beta$-allowable.
Let $C = (1+C_d)(C_d + 2C_\beta^{-1})$.  For $k \ge 1$, and all $\psi \in WV(Y)$,
\begin{eqnarray}
\| \hLp_{\Delta, \be} \psi \|_{\Delta_k} & \le & e^{-\gamma } \| \psi \|_{\Delta_{k-1}} \label{eq:k LY} \\
\bigvee_{\Delta_0} \hLp_{\Delta, \be} \psi & \le & \tfrac{4}{5} \| \psi \|_{WV} + C \int_{\hDelta^1_{\be}} |\psi| dm_{\Delta}
\label{eq:0 var} \\
| \hLp_{\Delta, \be} \psi |_{L^\infty(\Delta_0)} & \le & \tfrac{2}{5} \| \psi \|_{WV} + C_\beta^{-1}(1+C_d) \int_{\hDelta_{\be}^1} |\psi|
\label{eq:0 infty} \\
\int_{\Delta} |\hLp_{\Delta, \be}^n \psi| | \, dm_\Delta & \le & \int_{\hDelta_{\be}^n} |\psi | \, dm_\Delta
\; \mbox{ for all $n \ge 1$.}
\label{eq:L1 Delta}
\end{eqnarray}
\end{lemma}
The same estimates hold for $\Lp_\Delta$ with $\hDelta_{\be}^n$ replaced by $\Delta$.

\begin{proof}
Note that $Jf_\Delta(x, k) = 1$ if $k < \tau^{n_0}(x)-1$.  Thus if $k \ge 1$ and
$x \in \Delta_k$, then $\hLp_{\Delta, \be} \psi(x) = (1_{\hDelta^1_{\be}} \psi )\circ f_{\Delta}^{-1}(x)$.  

Moreover, for $J \in \cD_k$ and $k \ge 1$, we have 
$f_\Delta^{-1}(J) \subset J' \in \cD_{k-1}$, and by definition of $\cD_{k-1}$, $J'$ is either
disjoint from $H_\Delta$ or contained in it.  Thus $1_{\hDelta^1_{\be}}$ is either
identically 0 or 1 on $J'$ and so does not affect the variation.
With these points noted, \eqref{eq:k LY} holds trivially.
Similarly, \eqref{eq:L1 Delta} follows immediately from the definition of $\hLp_{\Delta, \be}$.

We proceed to prove \eqref{eq:0 var}.  The proof follows closely the proof of Proposition~\ref{prop:LY}.
As in Section~\ref{sec:uniform gap}, denote by $\cJ_{\be}$
the finite set of intervals in $Y$ that are the images of the one-cylinders
for $F_{\be}^{n_0}$.  Since we identify $\Delta_0$ with $Y$, we will abuse notation and
again denote the finite set of images $(\pi_\Delta|_{\Delta_0})^{-1}(\cJ_{\be})$ in $\Delta_0$
by simply $\cJ_{\be}$.  

Let $\cI_{\be}$ denote the countable collection of intervals in
$f_{\Delta}^{-1}(\Delta_0)$ so that for each $I_i \in \cI_{\be}$, $f_\Delta(I_i) = J$ for some
$J \in \cJ_{\be}$. 
We use the notation $I_i = (u_i, v_i)$, and $I_{i,k}$ denotes the intervals
in $\cI_{\be}$ that lie in $\Delta_k$.
Remark that each $I_{i,k} \subset J' \in \cD_k$ so that $I_{i,k}$ is either in $H_\Delta$ or
is disjoint from it.\footnote{Indeed, since a neighbourhood of $z$ cannot enter $Y$ in one step,
every $I_i \in \cI_{\be}$ is disjoint from $H_\Delta$ in this particular tower.}
Also, for $(x,k) \in I_{i,k}$, by definition $\tau^{n_0}(\pi_\Delta(x,0)) = k$, so that
\begin{equation}
\label{eq:translate}
Jf_\Delta(x,k) = |DF^{n_0}(\pi_{\Delta}(x,0))| = |Df^{\tau^{n_0}(x)}(x)| = |Df^k(x)| \, ,
\end{equation}
where for brevity, we denote $\pi_\Delta(x,0) = x$.

Now for $\psi \in WV(\Delta)$, following \eqref{eq:var split}, we obtain,
\begin{equation}
\label{eq:var split Delta}
\bigvee_{\Delta_0} \hLp_{\Delta, \be} \psi 
\le \sum_{I_i \in \cI_{\be}} \bigvee_{I_i} \frac{\psi}{Jf_\Delta}
+ \sum_{I_i \in \cI_{\be}}  \frac{|\psi(u_i)|}{Jf_\Delta(u_i)} + \frac{|\psi(v_i)|}{Jf_\Delta(v_i)} \, .
\end{equation}
For $I_i = I_{i,k} \subset \Delta_k$, using \eqref{eq:gamma} and \eqref{eq:translate}, we estimate the first term above as in \eqref{eq:first var term},
\[
\bigvee_{I_{i,k}} \frac{\psi}{Jf_\Delta} \le \frac{1}{2.5} \bigvee_{I_{i,k}} \psi e^{-\gamma k}
+ C_d(1+C_d) \int_{I_{i,k}} |\psi| \, .
\]
We estimate the second term in \eqref{eq:var split Delta} using \eqref{eq:sum} and the
large images property as in \eqref{eq:large image},
\begin{equation}
\label{eq:var second term}
\frac{|\psi(u_i)|}{Jf_\Delta(u_i)} + \frac{|\psi(v_i)|}{Jf_\Delta(v_i)}
\le \frac{1}{2.5} \bigvee_{I_{i,k}} \psi e^{-\gamma k}+ 2 C_\beta^{-1} (1+C_d) \int_{I_{i,k}} |\psi| \, dm_\Delta \, .
\end{equation}
Putting these estimates together in \eqref{eq:var split Delta} completes the proof of
\eqref{eq:0 var},
\[
\begin{split}
\bigvee_{\Delta_0} \hLp_{\Delta, \be} \psi & \le \frac{2}{2.5} \sum_{k \ge 1} \sum_{I_{i,k}} e^{-\gamma k} \bigvee_{I_{i,k}} \psi + (1+C_d)(C_d + 2C_\beta^{-1}) \int_{I_{i,k}} |\psi|  \\
& \le \frac{4}{5} \| \psi\|_{WV} + (1+C_d)(C_d + 2C_\beta^{-1}) \int_{\hDelta^1_{\be}} |\psi| \, .
\end{split}
\]

Finally, \eqref{eq:0 infty} follows immediately from \eqref{eq:var second term} since
for $x \in \Delta_0$,
\[
\begin{split}
|\hLp_{\Delta, \be} \psi(x) | & \le \sum_{I_i \in \cI_{\be}, y \in I_i} \frac{|\psi(y)|}{Jf_\Delta(y)}
\le \frac{1}{2.5} \sum_{i,k} \bigvee_{I_{i,k}} \psi e^{-\gamma k} + C_\beta^{-1}(1+C_d) \int_{I_{i,k}} |\psi| \, dm_\Delta \\
& \le \frac{2}{5} \| \psi \|_{WV} + C_\beta^{-1}(1+C_d) \int_{\hDelta^1_{\be}} |\psi| \, dm_{\Delta} \, .
\end{split}
\]
\end{proof}

\begin{lemma}
\label{lem:spectrum LDelta}
The operator $\Lp_\Delta$ has a spectral gap on $WV(\Delta)$.
\end{lemma}

\begin{proof}
Lemma~\ref{lem:tower LY} applied to $\Lp_\Delta$ implies that as an operator on 
$WV(\Delta)$, $\Lp_\Delta$ has 
spectral radius at most 1 and essential spectral radius at most $\max \{ e^{-\gamma}, 4/5 \}$.
Since $\Lp_\Delta^* m_\Delta = m_\Delta$, 1 is in the spectrum of $\Lp_\Delta$ so that $\Lp_\Delta$
is quasi-compact.  Indeed, \eqref{eq:L1} implies that the peripheral spectrum has no Jordan blocks.
Thus by \cite[III.6.5]{kato}, $\Lp_{\Delta}$ admits the following representation: there exist
finitely many eigenvalues $e^{i \theta_j}$, $j = 1, \ldots, N$ and finite-dimensional eigenprojectors $\Pi_j$
corresponding to $\theta_j$ such that 
\begin{equation}
\label{eq:first decomp}
\Lp_\Delta = \sum_{j=1}^N e^{i \theta_j} \Pi_j + \cR \, ,
\end{equation}
where $\cR$ has spectral radius strictly less than 1 and $\Pi_j \Pi_k = \Pi_j \cR = \cR \Pi_j = 0$ for all $j \ne k$.

Note that if $g \in WV(\Delta)$ satisfies $\Lp_\Delta g = g$, then $g_0 := g \circ (\pi_{\Delta}|_{\Delta_0})^{-1}$
is in $BV(Y)$ and $\Lp_0 g_0 = g_0$.  Since $\Lp_0$ has a spectral gap (see the proof of Theorem~\ref{thm:gap}), 
this implies that there can be at most one (normalised) fixed point for $\Lp_\Delta$.  
Thus the eigenvalue 1 is simple. 

Conversely, if $g_0$ denotes the unique element of $BV(Y)$ with $\int g_0 \, dm =1$ such that
$\Lp_0 g_0 = g_0$, then we may define $g_\Delta$ such that $\Lp_\Delta g_\Delta = g_\Delta$ by
\[
g_\Delta(x,k) = c_0 g_0 \circ \pi_\Delta(x,0), \; \mbox{ for each $(x,k) \in \Delta_k$ and $k \ge 0$,} 
\]
where $c_0$ is chosen so that $\int_\Delta g_0 \, dm_\Delta = 1$.  In particular, for each
$(x,k)$ and $k \ge 0$,
\begin{equation}
\label{eq:min on Delta}
c_0 c_g \le g_\Delta(x,k) = g_\Delta(x,0) \le c_0 C_g  \, ,
\end{equation}
where $C_g = |g_0|_\infty$ and $c_g>0$ is from \eqref{eq:lower g}.   We will use this fact to eliminate
the possibility of other eigenvalues of modulus 1.

It is convenient to first establish the following claim.
\begin{claim}
The peripheral spectrum of $\Lp_\Delta$ on $WV(\Delta)$ is cyclic: if $e^{i\theta}$ is an eigenvalue, then
so is $e^{i \theta n}$ for each $n \in \mathbb{N}$.
\end{claim}

\begin{proof}[Proof of Claim]
Since $\Lp_\Delta$ is a positive operator and $\Lp_\Delta^* m_\Delta = m_\Delta$, it follows from Rota's
Theorem \cite[Theorem~2.2.9]{BoyGor} (see also \cite[Theorem~1]{Schaefer}) that the peripheral
spectrum of $\Lp_\Delta$ on $L^1(m_\Delta)$ is cyclic.  It remains to show that this property holds as
well in $WV(\Delta)$.  We follow the strategy in \cite[proof of Theorem 6.1]{Keller}.
For $\theta \in [0, 2\pi)$, define 
\[
S_{n, \theta} = \frac{1}{n} \sum_{k=0}^{n-1} e^{-\theta k} \Lp_{\Delta}^k \, .
\]
It follows from \eqref{eq:first decomp} that $\lim_{n\to\infty} S_{n, \theta} = \Pi_j$ if $\theta = \theta_j$ and
$\lim_{n\to\infty} S_{n, \theta} = 0$ otherwise.  Indeed, the convergence in both cases is pointwise uniformly since
$\lim_{n \to \infty} \frac 1n \sum_{k=0}^{n-1} e^{(\theta_j - \theta)k} =0$ whenever $\theta \neq \theta_j$.
Then $\bigvee c \psi = |c| \bigvee \psi$ for any constant $c$ implies the sequence converges in $\| \cdot \|_{WV}$.

Since $\Lp_{\Delta}$ is an $L^1(m_\Delta)$ contraction (i.e. $\int |\Lp_\Delta \psi| \, dm_\Delta \le \int |\psi| \, dm_\Delta$),
then so is $S_{n,\theta}$, and since $WV(\Delta)$ is dense in $L^1(m_\Delta)$, $\lim_{n \to \infty} S_{n,\theta}$ 
extends to a bounded linear operator on $L^1(m_\Delta)$, with convergence in $\| \cdot \|_{L^1(m_\Delta)}$.  
Taking $\theta = \theta_j$, we may view
$\Pi_j : L^1(m_\Delta) \to \Pi_j(WV(\Delta))$.  

Now if $0 \neq f \in L^1(m_\Delta)$ satisfies $\Lp_\Delta f = e^{i \theta} f$, then
$0 \ne f = \lim_{n\to \infty} S_{n,\theta}$, which implies that $\theta = \theta_j$ for some $j$ and 
$f \in \Pi_j(WV(\Delta))$ is an element of $WV(\Delta)$.  Thus the peripheral spectra of $\Lp_\Delta$
on $WV(\Delta)$ and $L^1(\Delta)$ coincide. 
\end{proof}

Returning to the proof of the lemma,
suppose for the sake of contradiction that 1 is not the only eigenvalue of modulus 1.  Then according to the Claim,
there exists $h \in WY(\Delta)$ and 
$p,q \in \mathbb{N}\setminus \{0\}$ such that $\Lp_\Delta h = e^{i \pi p/q} h$.  It follows that
$h$ is complex-valued and that
$\int_{\Delta} \real(h) \, dm_{\Delta} = \int_{\Delta} \imag(h) \, dm_{\Delta} = 0$. 

Since $\Lp_\Delta^q h = h$, $h$ takes on all its possible values in the first $q$ levels,
$\cup_{k = 0}^{q-1} \Delta_k$.  In particular, 
$\sup_{\Delta} |\real(h)| = \sup_{\cup_{k=0}^{q-1} \Delta_k} |\real(h)|$,
and similarly for $\imag(h)$.  Thus by \eqref{eq:min on Delta}, we may choose
$\kappa > 0$ such that
\[
\psi := \kappa \real(h) + g_\Delta \; \mbox{ satisfies } \; \inf_\Delta \psi > 0.
\]
Note $\int_\Delta \psi \, dm_\Delta = 1$.  Next, for $s \in \mathbb{R}$, 
define 
\begin{equation}
\label{eq:s}
\psi_s = s \psi + (1-s) g_\Delta = s \kappa \real(h) + g_\Delta \, .  
\end{equation}
Note that $\psi_s$ also takes on all its possible values in 
$\cup_{k=0}^{q-1} \Delta_k$ and $\Lp_\Delta^q \psi_s = \psi_s$.

Let $\cS = \{ s \in \mathbb{R} : \essinf_{\Delta} \psi_s > 0 \}$.  By construction of $\psi$, and 
the compactness of $\cup_{k=0}^{q-1} \Delta_k$, $\cS$ contains $[0,1]$ and is open.

We will show that $\cS$ contains $\mathbb{R}^+$.  Suppose not.  Let $t > 1$ be an
endpoint of $\cS$  that is not in $\cS$.  Then $\essinf_\Delta \psi_t = \essinf_{\cup_{k=0}^{q-1} \Delta_k} \psi_t = 0$.  
Without loss of generality, we may work with a representative of 
$\psi_t$ that is lower semicontinuous.\footnote{We use here that any function of bounded
variation can be written as the difference of two monotonic functions so that one-sided limits
exist at each point \cite[Theorem~5, Section 5.2]{Roy}.} 
Since $\int_\Delta \psi_t \, dm_\Delta =1$, there must exist
$(y,j) \in \cup_{k=0}^{q-1} \Delta_k$ such that $\psi_t(y) > 0$.  By lower semicontinuity,
there exists an interval $A \subset \Delta_j$ such that $\inf_A \psi_t := a > 0$.

Since $f_\Delta$ is topologically mixing, we can find
$N \in \mathbb{N}$ such that $f_\Delta^{N+i}(A) \supseteq \cup_{k=0}^{q-1} \Delta_k$,
for $i=0, \ldots, q-1$.  One of these iterates must equal $nq$ for some $n \in \mathbb{N}$.
Thus for any $x \in \cup_{k=0}^{q-1} \Delta_k$, 
\[
\psi_t(x) = \Lp_\Delta^{nq} \psi_t(x) \ge \frac{a}{\sup_A |Df^{nq} \circ \pi_\Delta| } > 0 \, ,
\]
since $\sup |Df| < \infty$ and $n$ is fixed.  This proves that $t \in \cS$ so in fact $\mathbb{R}^+ \subset \cS$.
By \eqref{eq:s}, this implies that $\real(h) \ge 0$, but since $\int_\Delta \real(h) \, dm_\Delta= 0$, 
it must be that
$\real(h) \equiv 0$.  A similar argument forces $\imag(h) \equiv 0$, providing the needed
contradiction. 
\end{proof}

As in \eqref{eq:tri}, denote by $||| \cdot |||$ the norm which views $\Lp_\Delta$ as
an operator from $WV(\Delta)$ to $L^1(m_\Delta)$.

\begin{lemma}
\label{lem:Delta pert}
There exists $C>0$ such that for any $\be >0$, $||| \Lp_\Delta - \hLp_{\Delta, \be} ||| \le C \mu(H_{\be}(z))^{1-\gamma/\zeta}$.
\end{lemma}
\begin{proof}
Let $\psi \in WV(\Delta)$.  Then,
\[
\int_\Delta |(\Lp_\Delta - \hLp_{\Delta, \be} )\psi | \, dm_\Delta \le \int_{H_\Delta} |\psi| \, dm_\Delta
\le \| \psi \|_{WV} \sum_{k \ge 1} e^{\gamma k} m_\Delta(H_\Delta \cap \Delta_k) \, .
\]
This expression can be made small in $\mu(H_{\be}(z))$ as follows.  Let $d\mu_\Delta = g_\Delta dm_\Delta$.
Then $(\pi_\Delta)_*\mu_\Delta = \mu$, so that using \eqref{eq:min on Delta},
\[
\begin{split}
\sum_{k \ge 1} e^{\gamma k} & m_\Delta(H_\Delta \cap \Delta_k)
 = \sum_{k = 1}^{- \zeta^{-1} \log \mu(H_{\be}(z))} e^{\gamma k} m_\Delta(H_\Delta \cap \Delta_k)
+ \sum_{k \ge - \zeta^{-1} \log \mu(H_{\be}(z)) } e^{\gamma k} m_\Delta(H_\Delta \cap \Delta_k) \\
& \le \mu(H_{\be}(z))^{-\gamma/\zeta} (c_0c_g)^{-1} \sum_{k = 1}^{- \zeta^{-1} \log \mu(H_{\be}(z))} \mu_\Delta(H_\Delta \cap \Delta_k)
+ \sum_{k \ge - \zeta^{-1} \log \mu(H_{\be}(z)) } C e^{(\gamma - \zeta)k} \\
& \le C \mu(H_{\be}(z))^{1-\gamma/\zeta} \, .
\end{split}
\]
\end{proof}

With these elements in place, we are ready to prove Proposition~\ref{prop:tower gap}.

\begin{proof}[Proof of Proposition~\ref{prop:tower gap}]
For fixed $\beta>0$, the Lasota-Yorke inequalities in Lemma~\ref{lem:tower LY} have uniform constants.  Thus
the spectral radius of $\hLp_{\Delta, \be}$ has essential spectral radius at most $\max \{ e^{-\gamma} , \frac{4}{5} \}$
for all $\beta$-allowable holes.

This,
together with Lemma~\ref{lem:Delta pert} implies by \cite[Corollary~1]{KL pert} that the spectra and spectral projectors
of $\hLp_{\Delta, \be}$ outside the disk of radius $\max \{ e^{-\gamma} , \frac{4}{5} \}$ vary 
H\"older continuously in $\mu(H_{\be}(z))$.  Thus there exists $\ve_\beta(\Delta)>0$ such that for all
$\beta$-allowable holes with $\be < \ve_\beta(\Delta)$, the operators $\hLp_{\Delta, \be}$ enjoy a uniform spectral gap
and can be decomposed as in the statement of the proposition. 
\end{proof}

Our final lemma of this section demonstrates that the spectral radius of $\hLp_{\Delta, \be}$ yields the
escape rate from both $\Delta$ and $I$.

\begin{lemma}
\label{lem:spec escape}
Under the hypotheses of Proposition~\ref{prop:tower gap}, 
$- \log \lambda_{\Delta, \be} = \mathfrak{e}(H_{\be}(z))$, where $\mathfrak{e}(H_{\be}(z))$ is from
\eqref{eq:escape f}.
\end{lemma}

\begin{proof}
Using Proposition~\ref{prop:tower gap}, we compute
\[
\begin{split}
- \mathfrak{e}(H_{\be}(z)) & = \lim_{n \to \infty} \frac 1n \log \mu(\cap_{i=0}^{n-1} f^{-i}(I \setminus H_{\be}))
= \lim_{n \to \infty} \frac 1n \log \mu_{\Delta}(\hDelta^n_{\be}) \\
& = \lim_{n \to \infty} \frac 1n \log \int_\Delta \hLp^n_{\Delta, \be}(g_\Delta) \, dm_{\Delta}
= \lim_{n \to \infty} \frac 1n \log \left( \lambda_{\Delta, \be}^n \he_{\Delta, \be}(g_\Delta) + \int_\Delta \cR_{\Delta, \be}^n (g_\Delta) \, dm_\Delta \right) \\
& = \log \lambda_{\Delta, \be} \, ,
\end{split}
\]
since $\he_{\Delta, \be}(g_\Delta) > 0$ due to \eqref{eq:min on Delta}. 
\end{proof}


\subsection{Hitting time statistics for $\beta$-allowable holes}
\label{sec:HTS beta}

To prove Theorem~\ref{thm:HTS}, we will compute the following limit for fixed $\alpha > 0$
amd $t>0$,
\[
\lim_{\be \to 0} \frac{-1}{t \mu(H_{\be}(z))^{1-\alpha}} \log \mu \left( r_{H_{\be}(z)} > \frac{t}{\mu(H_{\be}(z))^\alpha} \right) \, .
\]
Recall that $\alpha > 0$ was fixed at the beginning of Section~\ref{sec:nonMarkov tower}
and affected the chosen value of $\gamma$ via \eqref{eq:gamma two}.

Since $\pi_\Delta \circ f_\Delta = f \circ \pi_\Delta$, $\pi_\Delta(H_\Delta) = H_{\be}(z)$ 
then $r_\Delta := r_{H_\ve(z)} \circ \pi_\Delta$ defines the first hitting time to $H_\Delta$.  Then since 
$(\pi_\Delta)_*\mu_\Delta = \mu$, it is equivalent to estimate,
\[
\lim_{\be \to 0} \frac{-1}{t \mu_\Delta(H_\Delta)^{1-\alpha}} \log \mu_\Delta \left( r_{\Delta} > \frac{t}{\mu_\Delta(H_\Delta)^\alpha} \right) \, .
\]
Setting $n_{\be} = \lfloor t \mu_\Delta(H_\Delta)^{-\alpha} \rfloor = \lfloor t \mu(H_{\be}(z))^{-\alpha} \rfloor$, we estimate as in \cite[Section~2.5]{BruDemTod18},
\[
\begin{split}
\mu_{\Delta}(r_\Delta > n_{\be} ) & = \int_{\hDelta^{n_{\be}}_{\be}} g_\Delta \, dm_\Delta
= \int_{\Delta} \hLp^{n_{\be}+1}_{\Delta, \be} g_\Delta \, dm_\Delta \\
& = \lambda_{\Delta, \be}^{n_{\be}+1} \int_\Delta \lambda_{\Delta, \be}^{-n_{\be}-1} \hLp_{\Delta, \be}^{n_{\be}+1} (g_\Delta - \hg_{\Delta, \be} ) \, dm_\Delta + \lambda_{\Delta, \be}^{n_{\be}+1} \int_\Delta \hg_{\Delta, \be} \, dm_\Delta \, ,
\end{split}
\]
where $\hg_{\Delta, \be}$ is from Proposition~\ref{prop:tower gap}.  Thus,
\[
\log \mu_\Delta(r_\Delta > n_{\be}) = (n_{\be}+1) \lambda_{\Delta, \be} + \log \left(
1 + \int_\Delta \lambda_{\Delta, \be}^{-n_{\be}-1} \hLp_{\Delta, \be}^{n_{\be}+1} (g_\Delta - \hg_{\Delta, \be} ) \, dm_\Delta
\right) \, .
\]
Dividing by $- t \mu_\Delta(H_\Delta)^{1-\alpha}$, we see that the first term becomes simply
$- \frac{ \log \lambda_{\Delta, \be}}{\mu_\Delta(H_\Delta)}$. 
Since $-\log \lambda_{\Delta, \be} = \mathfrak{e}(H_{\be}(z))$ by Lemma~\ref{lem:spec escape},
the first term yields esc$(z)$, which is either 1 or  $1 - \lambda_z^{-1/\ell}$ as needed,
as $\be \to 0$ according to Theorem~\ref{thm:spike}.
It remains to show that the second term tends to 0 as $\be \to 0$.

Using the spectral decomposition in Proposition~\ref{prop:tower gap}, we define
$c_\ve = \he_{\Delta, \be}(g_\Delta)$ and write
\[
\lambda_{\Delta, \be}^{-n_{\be}-1} \hLp_{\Delta, \be}^{n_{\be}+1} (g_\Delta - \hg_{\Delta, \be} )
= (c_\ve -1) \hg_{\Delta, \be} + \lambda_{\Delta, \be}^{-n_{\be}-1} \cR_{\Delta, \be}^{n_{\be}+1} g_\Delta \, .
\]
Integrating this equation, we see that we must estimate,
\begin{equation}
\label{eq:next step}
\log \left(
1 + \int_\Delta \lambda_{\Delta, \be}^{-n_{\be}-1} \hLp_{\Delta, \be}^{n_{\be}+1} (g_\Delta - \hg_{\Delta, \be} ) \, dm_\Delta
\right)
= \log \left( c_\ve + \int_\Delta \lambda_{\Delta, \be}^{-n_{\be}-1} \cR_{\Delta, \be}^{n_{\be}+1} g_\Delta \, dm_\Delta \right) \, .
\end{equation}
Again using Proposition~\ref{prop:tower gap}, we bound the integral by,
\[
\left| \int_\Delta \lambda_{\Delta, \be}^{-n_{\be}-1} \cR_{\Delta, \be}^{n_{\be}+1} g_\Delta \, dm_\Delta \right|
\le A_\beta e^{-\sigma_\beta (n_{\be}+1)} \| g_\Delta \|_{WV} \le C e^{-\sigma_\beta t \mu(H_{\be}(z))^{-\alpha}} \, ,
\]
and this quantity is super-exponentially small in $\mu(H_{\be}(z))$.  By Lemma~\ref{lem:Delta pert} and
\cite[Corollary~1]{KL pert}, 
\[
|c_\ve -1 | = |\he_{\Delta, \be}(g_\Delta) - e_{\Delta}(g_\Delta)| \le C \mu(H_{\be}(z))^{1-\gamma/\zeta} \log \mu(H_{\be}(z))^{-1} \, .
\]
Putting these estimates together in \eqref{eq:next step} and dividing by $t \mu(H_{\be}(z))^{1-\alpha}$, we obtain
\[
\begin{split}
\lim_{\be \to 0} & \frac{1}{t \mu(H_{\be}(z))^{1-\alpha}} \log \left( 1 + \O(-\mu(H_{\be}(z))^{1-\gamma/\zeta} \log \mu(H_{\be}(z)) \right) \\
& = \lim_{\be \to 0} \frac{1}{t} \O \Big( -\mu(H_{\be}(z))^{\alpha-\gamma/\zeta} \log \mu(H_{\be}(z)) \Big) \, ,
\end{split}
\]
which tends to 0 since $\alpha > \gamma/\zeta$ by \eqref{eq:gamma two}.
The above limit $\be \to 0$ is understood to be taken along sequences of $\beta$-allowable holes.


\subsection{Proof of Theorem~~\ref{thm:HTS} via approximation when $z \in \mbox{orb}(f(c))$}
\label{sec:beta approx two}

Section~\ref{sec:HTS beta} proves Theorem~\ref{thm:HTS} 
when $z \in \mbox{orb}(f(c))$ for each $\alpha >0$ and 
$\beta>0$ along sequences $(\be_n)_n$ where each $H_{\be_n}$ is $\beta$-allowable.
It remains to consider the alternative case when $\alpha > 0$ is still fixed and 
we have to approximate 
 a given sequence $H_{\be_n}$  by $\beta$-allowable $H_{\be_n'}$.  
 The approximation follows closely the strategy in Section~\ref{sec:beta approx}.  
 As in that section, we first present the argument in the case that $z \in \mbox{orb}(f(c))$ is periodic.

Recall that if $H_{\be}$ is $\beta$-allowable, then it is also $\beta'$-allowable for any $\beta' < \beta$, 
so as in Section~\ref{sec:beta approx}, we take our approximating sequence with $\beta$ tending to 0.  As before, we assume
$\beta < (2\lambda_z)^{-1}$.

Using precisely the same discussion and notation as in Section~\ref{sec:beta approx}, 
we suppose that each $\ve$ that corresponds to a non-$\beta$-left-allowable hole satisfies
$z-\ve \in (c_{i,j,k+1}-\beta v_{i,j,k}, c_{i,j,k+1} + \beta v_{i,j,k})$ 
 for some $i,j,k$.
We approximate $\ve$ from above by
$\ve_o^L := z - (c_{i,j,k+1} - \beta v_{i,j,k})$ and from below by
$\ve_u^L := z - (c_{i,j,k+1} + \beta v_{i,j,k})$.  Both $H_{\ve_o^L}$ and $H_{\ve_u^L}$ are $\beta$-left-allowable
 and $\ve \in (\ve_u^L, \ve_o^L)$. 
Applying \eqref{eq:ratio bound}, we have
\[
\frac{\ve_u^L}{\ve_o^L} \gtrsim 1 - 2 \beta \lambda_z \; \mbox{ and } \;
\frac{\ve_o^L}{\ve_u^L} \lesssim  1 + 2 \beta \lambda_z \, .
\] 
The right hand estimates for
the analogous $\ve_u^R$ and $\ve_o^R$ enjoy similar bounds.

To complete the proof of Theorem~\ref{thm:HTS}, we consider the following limit as $\ve \to 0$
 for fixed $t, \alpha>0$,
\begin{equation}
\label{eq:HTS limit}
\frac{-1}{t \mu(H_\ve(z))^{1-\alpha}} \log\mu\left( r_{H_{\ve}(z)} > \frac{t}{\mu(H_\ve(z))^\alpha }\right) \, .
\end{equation}
We first estimate this from below.  Let $r_u$ denote the first hitting time to the smaller set 
$(z-\ve^L_u, z+\ve^R_u) \subset H_\ve(z)$.  Note that $r_u > r_{H_\ve(z)}$, and
$\frac{\mu(z-\eps_u^L, z+\eps_u^R)}{\mu(z-\eps, z+\eps)} \ge C_\ve (1- 2\beta \lambda_z)^{\frac1\ell}$
where $C_\ve \to 1$ as $\ve \to 0$. 

Setting $s = t (C_\ve (1- 2\beta \lambda_z)^{\frac1\ell})^\alpha$
we estimate \eqref{eq:HTS limit} from below by,
\[
\begin{split}
& \frac{\mu(z-\eps_u^L, z+\eps_u^R)^{1-\alpha}}{\mu(z-\eps, z+\eps)^{1-\alpha}} 
\frac{-(C_\ve (1-2 \beta \Lambda_z)^{1/\ell})^\alpha}{s \mu(z-\eps_u^L, z+\eps_u^R)^{1-\alpha}} \log\mu\left(r_u > \frac{t}{\mu(H_{\ve}(z))^\alpha } \right) \\
& \ge (C_\ve (1-2 \beta \Lambda_z)^{1/\ell})
\frac{-1}{s \mu(z-\eps_u^L, z+\eps_u^R)^{1-\alpha}} \log\mu\left(r_u > \frac{s}{\mu(z-\ve_u^L, z+\ve^R)^\alpha } \right) \, .\end{split}
\]
Taking the limit as $\ve \to 0$ yields a lower bound of
\[
( 1 - 2\beta \lambda_z)^{\frac1\ell} \cdot \left( 1 - \lambda_z^{-1/\ell} \right) \, ,
\] 
where the second factor comes from the application of Theorem~\ref{thm:HTS} to $\beta$-allowable holes in
the case that $z \in \mbox{orb}(f(c))$ is periodic.  

Similarly, one obtains an upper bound for \eqref{eq:HTS limit} of
$( 1 + 2\beta \lambda_z)^{\frac1\ell} \cdot \left( 1 - \lambda_z^{-1/\ell} \right)$.
Since these bounds hold for all sufficiently small $\beta$, we take $\beta \to 0$ to obtain the required
limit for Theorem~\ref{thm:HTS} along an arbitrary sequence $( \ve_n )_n$.

Finally, if $z\in \mbox{orb}(f(c))$ is preperiodic then the above calculations all go through similarly.
As in Section~\ref{sec:beta approx}, from the construction of $(Y, F)$ in Section~\ref{ssec:ratiopreper}, 
the periodic structure of the postcritical orbit can be pulled back to $z$ to generate the $(a_i)_i, \ (b_i)_i$ required, 
but the resulting bounds are of the form $( 1 \pm 2\beta \lambda)$ where $\lambda = |Df^p(f^{k_0}(z))|$.
Thus they tend to 1 as $\beta \to 0$, as required.


\subsection{Proof of Theorem~\ref{thm:HTS} when $z \notin \mbox{orb}(f(c))$}
\label{ssec:veryend}

We explain here how to adapt the results of \cite[Section~4.2.1]{BruDemTod18} to
achieve the required limit in Theorem~\ref{thm:HTS} for any $z \notin \mbox{orb}(f(c))$.
Since $f$ is Misiurewicz, 
one can choose an interval $Y$ containing $z$ whose endpoints are 
two points of a periodic orbit orb$(p)$ where orb$(p)$ is disjoint from orb$(z)$ and
the interior of $Y$.  If we define $F$ to be the induced map with first return time $\tau$,
then $F$ is a full-branched Gibbs-Markov map, which satisfies the conditions of
\cite[Theorem~2.1]{BruDemTod18}.  In particular, we consider the parameter
$n_1$ from \cite[eq. (2.1)]{BruDemTod18} to be chosen:  this is chosen so that $F^{n_1}$
has sufficient expansion.

At this point, we find it convenient to consider separately two cases:  $z$ is a recurrent point 
(every $\ve$-neighbourhood of $z$ contains a point in orb$(f(z))$); or $z$ is a nonrecurrent point.

{\em Case 1:} $z$ is a recurrent point.  By choice of $\partial Y$, $z$ is necessarily contained
in the interior of a domain $Y_i^k$ of $F^k$ for each $k \ge1$.  Thus $z \in Y_{\text{cont}} := \{ y \in Y : F^k \text{ is continuous at $y$ for all $k \ge 1$}\}$.  Moreover, for any sufficiently small $\ve$, $(z-\ve, z+\ve) \subset Y_i^{n_1}$
and so the lengths of images of intervals of monotonicity for $\hF_{\ve'}^{n_1}$, where
$\hF_{\ve'} := F|_{Y\setminus (z-\ve', z+\ve')}$,
have a positive uniform lower bound for all $\ve' < \ve$.  This ensures that
condition {\bf (U)} of \cite{BruDemTod18} is satisfied.  Thus we may apply
\cite[Theorem~3.2]{BruDemTod18} to conclude that $L_{\alpha, t}(z) = 1$.

{\em Case 2:} $z$ is not a recurrent point.  If an accumulation point of orb$(z)$ lies in $Y$,
then $z$ lies in the interior of a domain of $F^k$ for each $k$ and by the above argument, 
\cite[Theorem~3.2]{BruDemTod18} applies so that Theorem~\ref{thm:HTS} follows.

If, on the other hand, no accumulation points of orb$(z)$ lie in $Y$, then since
$\partial Y$ is periodic, $z$ is necessarily
an accumulation point of domains $\{ Y_i \}_i$ of $F$.  In this case, a modification of
the approach of \cite{BruDemTod18} is needed on two points.

First, fixing $\beta>0$, we only consider values of $\ve_L$ and $\ve_R$ so that
$z-\ve_L$ and $z+\ve_R$ are $\beta$-deep in intervals of monotonicity for $F^{n_1}$
around $z$.  These constitute $\beta$-allowable holes $(z-\ve_L, z+\ve_R)$
so that the uniformly
large images property {\bf (U)} of \cite{BruDemTod18} applies to
the punctured induced map, $\hF^{n_1}_{\be}$.  In particular, under these conditions, the
associated punctured transfer operators enjoy a uniform spectral gap in $BV(Y)$ for
all sufficiently small $\beta$-allowable holes.

Applying the results of \cite{KL zero} as in \cite[Section~2.3]{BruDemTod18}, we see that
esc$(z)=1$ as long as
\[
\lim_{\be \to 0} \frac{\mu(E^k_{\be})}{\mu(H_{\ve}(z))} = 0 \; \mbox{for each $k \ge 0$},
\]
where 
\[
E^k_{\be} = \{ y \in H_{\be}(z) : F^i(y) \notin H_{\be}(z), i = 1, \ldots, k, \mbox{ and } F^{k+1}(y) \in H_{\be}(z) \} \, .  
\]
Since $F$ is full branched, each domain $Y_i^k$ of $F^k$ has an interval which maps onto
$H_{\be}(z)$.  However, if $Y^k_i \subset H_{\be}(z)$, then $\tau(Y^k_i) \ge \log (|Y| \ve^{-1})/\log |Df|_\infty \ge C_0 \log \ve^{-1}$, where $\ve = \max\{ \ve_L, \ve_R\}$.  This implies that the
contribution to $E^k_{\be}$ from the collection of such intervals is dominated by
\[
\sum_{j \ge C_0 \log \ve^{-1}} C_k |H_{\be}(z)| \lambda_{\text{per}}^{-j}
\le C'_k \ve^{C_0 \lambda_{\text{per}}} |H_{\ve}(z)| \, ,
\]
where we have applied Proposition~\ref{prop:Str} since $F$ is full branched.
Since the invariant measure $\mu$ has density at $z$ bounded away from 0 and $\infty$, we
estimate,
\[
\frac{\mu(E^k_{\be})}{\mu(H_{\ve}(z))} \le C''_k \ve^{C_0 \lambda_{\text{per}}} \to 0
\; \mbox{ as $\ve \to 0$ for each $k$.}
\]
With these modifications, esc$(z)=1$ and \cite[Theorem~3.2]{BruDemTod18} implies the desired 
limit $L_{\alpha, t}(z)=1$ as well along sequences of $\beta$-allowable holes.

The approximation of more general holes $(z-\ve, z+\ve)$ by $\beta$-allowable holes
in order to prove the required limit 1 for Theorem~\ref{thm:HTS}
proceeds as in Section~\ref{sec:beta approx two}.  The case here is simpler since there
is no density spike so we do not need to maintain bounded ratios 
$\frac{\ve^L_u}{\ve^R_u}$, $\frac{\ve^L_o}{\ve^R_o}$
 for the approximating holes.



\begin{thebibliography}{XXX}

\bibitem[BY]{BuYu11} L.A.~Bunimovich and A.~Yurchenko, \emph{Where to place a hole to achieve a
maximal escape rate}, Israel J. of Math. {\bf 182} (2011), 229--252.

\bibitem[AHV]{atnip}  J.~Atnip, N.~Hayden, S.~Vaienti, \emph{Extreme value theory with spectral techniques: application to a simple attractor,} Discrete Contin. Dyn. Syst. \textbf{45} (2025), 3512--3544.

\bibitem[AFV]{AytFreVai15} H.\ Aytac, J.M.\ Freitas, S.\ Vaienti, 
\emph{Laws of rare events for deterministic and random dynamical systems,} Trans. Amer. Math. Soc. \textbf{367} (2015), 8229--8278.

\bibitem[BF]{BanFre23}  D.\ Bansard-Tresse, J.M.\ Freitas, 
\emph{Inducing techniques for quantitative recurrence and applications to {M}isiurewicz maps and doubly intermittent maps, } 
To appear in Ann. Inst. Henri Poincaré Probab. Stat.

\bibitem[BG]{BoyGor} A.~Boyarsky, P.~Gora, \emph{Laws of Chaos:  Invariant Measures and Dynamical Systems in One Dimension}, Probability and Its Applications, Birkh\"auser Boston 1997. 

\bibitem[BDM]{BDM} H.~Bruin, M.F.~Demers, I.~Melbourne, \emph{Existence and convergence properties of physical 
measures for certain dynamical systems with holes}, Ergodic Theory Dynam. Systems  {\bf 30} (2010), 687--728.

\bibitem[BDT]{BruDemTod18} H.\ Bruin, M.F.\ Demers, M.\ Todd, 
\emph{Hitting and escaping statistics: mixing, targets and holes,}
 Adv. Math. \textbf{328} (2018) 1263--1298.

\bibitem[DT1]{DemTod17} M.F.~Demers, M.~Todd, \emph{Slow and fast escape for open intermittent maps},
Comm. Math. Phys. {\bf 351} (2017), 775--835.

\bibitem[DT2]{DemTod21} M.F.~Demers, M.~Todd,  \emph{ Asymptotic escape rates and limiting distributions for multimodal maps,}  Ergodic Theory Dynam. Systems \textbf{41} (2021), 1656--1705.

 \bibitem[FP]{FerPol12} A.~Ferguson, M.~Pollicott, \emph{Escape rates for Gibbs measures},
  Ergodic Theory Dynam. Systems, {\bf 32} (2012), 961--988.

\bibitem[FFT1]{FreFreTod12} A.C.M.~Freitas, J.M.~Freitas, M.~Todd, \emph{The extremal index, hitting time statistics and periodicity,} Adv. Math. \textbf{231} (2012), 2626--2665.

\bibitem[FFT2]{FreFreTod15} A.C.M.~Freitas, J.M.~Freitas, M.~Todd, \emph{Speed of convergence for laws of rare events and escape rates,}
 Stochastic Process. Appl. \textbf{125} (2015), 1653--1687.

\bibitem[FFTV]{FreFreTodVai16} A.C.M.~Freitas, J.M.~Freitas, M.~Todd, S.\ Vaienti, 
\emph{Rare events for the Manneville-Pomeau map,} Stochastic Process. Appl. \textbf{126} (2016), 3463--3479.

\bibitem[GKV]{giulietti}  P.~Giulietti, P.~Koltai and S.~Vaienti, \emph{Targets and holes,}
Proceedings of AMS {\bf 149} (2021), 3293--3306.

\bibitem[HY]{HayYan20} N.\ Haydn, F.\ Yang,
\emph{ Local escape rates for $\phi$-mixing dynamical systems,}
 Ergodic Theory Dynam. Systems \textbf{40} (2020), 2854--2880.
 
 \bibitem[Ka]{kato} T.~Kato, \emph{Perturbation Theory for Linear Operators,} Springer-Verlag, Berlin-Heidelberg-New York 1976.
 
 \bibitem[K]{Keller} G.~Keller, \emph{Markov extensions, zeta functions and Fredholm theory for piecewise
 invertible dynamical systems}, Trans. Amer. Math. Soc. {\bf 314} (1989), 433--497.

\bibitem[KL1]{KL pert} G.~Keller, C.~Liverani, 
	\emph{Stability of the spectrum for transfer operators,} 
Annali della Scuola Normale Superiore di Pisa, 
Classe di Scienze (4) Vol. XXVIII, (1999), 141--152.
	
\bibitem[KL2]{KL zero} G.~Keller, C.~Liverani, 
\emph{Rare events, escape rates and quasistationarity: some exact formulae,} 
Journal of Statistical Physics {\bf 135} (2009), 519--534.  

\bibitem[LSV]{LSV} C.~Liverani, B.~Saussol, S.~Vaienti, \emph{Conformal measure
and decay of correlations for covering weighted systems,}  Ergodic Theory Dynam. Systems {\bf 18}
(1998), 1399--1420.

\bibitem[MS]{MelStr93} W.\ de Melo, S.\ van Strien, 
One-dimensional dynamics, Ergebnisse der Mathematik und ihrer Grenzgebiete (3), vol. 25, Springer-Verlag, Berlin, 1993.

\bibitem[M]{Mis81} M.\ Misiurewicz, 
\emph{Absolutely continuous measures for certain maps of an interval,}
 Inst. Hautes \'Etudes Sci. Publ. Math. \textbf{53} (1981), 17--51.
 
  \bibitem[N]{Now93} T.\ Nowicki,
 \emph{Some dynamical properties of {S}-unimodal maps,}
  Fund. Math. \textbf{142} (1993), 45--57.
 
 \bibitem[NS]{NowSan98} T.\ Nowicki, D.\ Sands, 
 \emph{Non-uniform hyperbolicity and universal bounds for S-unimodal maps,}
  Invent. Math. \textbf{132} (1998) 633--680.
 
 \bibitem[PU]{PolUrb18} M.~Pollicott, M.~Urbanski, \emph{Open Conformal Systems and Perturbations of Transfer Operators,} Lecture
Notes in Mathematics {\bf 2206}, Springer: Berlin, 2018.

\bibitem[R]{Roy} H.L.~Royden, \emph{Real Analysis}, Third Edition, Prentice Hall, (1988).
 
 \bibitem[S]{Schaefer} H.H.~Schaefer, \emph{On the point spectrum of positive operators,} Proc. Amer. Math. Soc.
 {\bf 15} (1964), 56--60.
 

\end{thebibliography}
\end{document}